\documentclass[11pt,reqno, twoside]{amsart}
\usepackage{amscd}
\usepackage{amsfonts}
\usepackage{amsmath}
\usepackage{amssymb}
\usepackage{amsthm}
\usepackage{fancyhdr}
\usepackage{latexsym}
\usepackage[colorlinks=true, pdfstartview=FitV, linkcolor=blue, citecolor=blue, urlcolor=blue]{hyperref}
\usepackage{enumitem}      
\usepackage{mathtools}            
\usepackage{upgreek}
\usepackage{indentfirst} % for indentation after chapter, section, subsection etc.        
\usepackage{schemata} %for brackets around paragraphs
\usepackage{blindtext, rotating}   % for rotating text or images
\usepackage{soul} % for striking through text
\setstcolor{red}
\usepackage{color}
\usepackage{tikz}
\usetikzlibrary{arrows.meta}
\usetikzlibrary{decorations.markings}
\tikzset{->-/.style={decoration={
  markings,
  mark=at position #1 with {\arrow{>}}},postaction={decorate}}}
  \tikzset{middlearrow/.style={
        decoration={markings,
            mark= at position 0.55 with {\arrow{#1}} ,
        },
        postaction={decorate}
    }
}
\usepackage{caption}           
\newcommand{\eee}[1]{\begin{equation}#1\end{equation}}
\newcommand{\sss}[1]{\begin{subequations}#1\end{subequations}}

\newcommand{\ddd}[1]{\begin{alignat}{2}#1\end{alignat}}

\newcommand{\nn}{\nonumber}
\newcommand{\p}{\partial}

\definecolor{ddgreen}{RGB}{0,170,0}

\renewcommand{\b}{\mathcolor{blue}}

\newcommand{\no}[1]{\left\| #1 \right\|}

\usepackage{scalerel}
\usepackage{stackengine}
\stackMath
\def\widehatgap{-5.5pt}
\def\subdown{-3.3pt}
\newcommand\what[2][]{%
\renewcommand\stackalignment{l}%
\stackon[\widehatgap]{#2}{%
\stretchto{%
    \scalerel*[\widthof{$#2$}]{\kern-.6pt\text{\textasciicircum}\kern-1.1pt}%
    {\rule[-0.8\textheight]{1ex}{\textheight}}%WIDTH-LIMITED BIG WEDGE
}{2ex}% THIS SQUEEZES THE WEDGE TO 0.5ex HEIGHT
_{\smash{\belowbaseline[\subdown]{\scriptscriptstyle#1}}}%
}}

\makeatletter
\renewcommand\subsection{\@startsection{subsection}{2}%
  \z@{-.5\linespacing\@plus-.7\linespacing}{.5\linespacing}%
  {\normalfont\bfseries}}
\renewcommand\subsubsection{\@startsection{subsubsection}{3}%
  \z@{.5\linespacing\@plus.7\linespacing}{-.5em}%
  {\normalfont\scshape}}
\makeatother
\makeatletter
\def\mathcolor#1#{\@mathcolor{#1}}
\def\@mathcolor#1#2#3{%
\protect\leavevmode
\begingroup
\color#1{#2}#3%
\endgroup
}
\makeatother
\theoremstyle{plain}  % default
\newtheorem{theorem}{Theorem}[section]

\newtheorem{lemma}{Lemma}[section]

\theoremstyle{definition}

\newtheorem{remark}{Remark}[section]
\newenvironment{Proof}[1][\proofname]
{\proof[\textnormal{\textbf{#1.}}]}{\endproof}
\newcommand{\bp}{\begin{Proof}}
\newcommand{\ep}{\end{Proof}}

\DeclareMathSizes{12}{12}{7}{5}
\numberwithin{figure}{section}
\numberwithin{equation}{section}
\usepackage{geometry}
\makeatletter
\def\l@section{\@tocline{1}{0pt}{1pc}{}{}}
\def\l@subsection{\@tocline{2}{0pt}{1pc}{4.6em}{}}
\def\l@subsubsection{\@tocline{3}{0pt}{1pc}{7.6em}{}}
\renewcommand{\tocsection}[3]{%
  \indentlabel{\@ifnotempty{#2}{\makebox[2.3em][l]{%
    \ignorespaces#1 #2.\hfill}}}#3}
\renewcommand{\tocsubsection}[3]{%
  \indentlabel{\@ifnotempty{#2}{\hspace*{2.3em}\makebox[2.3em][l]{%
    \ignorespaces#1 #2.\hfill}}}#3}
\renewcommand{\tocsubsubsection}[3]{%
  \indentlabel{\@ifnotempty{#2}{\hspace*{4.6em}\makebox[3em][l]{%
    \ignorespaces#1 #2.\hfill}}}#3}
\makeatother 
\begin{document}
\title{Well-Posedness of Initial-Boundary Value Problems
\\
 for a Reaction-Diffusion Equation}
\author{A. Alexandrou  Himonas, Dionyssios Mantzavinos$^*$ \& Fangchi Yan}
\date{October 11, 2018. \mbox{}$^*$\!\textit{Corresponding author}:
mantzavinos@ku.edu}
\begin{abstract}
A reaction-diffusion equation with power nonlinearity formulated either on the half-line or on the finite interval with nonzero boundary conditions is shown to be locally well-posed in the sense of Hadamard for data in Sobolev spaces.
The result is established via a contraction mapping argument, taking advantage of a novel approach that utilizes the formula produced by the unified transform method of Fokas for the forced linear heat equation to obtain linear estimates analogous to those previously derived for the nonlinear Schr\"odinger, Korteweg-de Vries and ``good'' Boussinesq equations. Thus, the present work extends the recently introduced ``unified transform method approach to well-posedness'' from dispersive equations to diffusive ones.
\end{abstract}
\maketitle
\markboth
{Well-Posedness of Initial-Boundary Value Problems for a Reaction-Diffusion Equation}
{A. Himonas, D. Mantzavinos \& F. Yan}
%

%
%
%
%%%%%%%%%%%%%%%%%%%%%
%
%	    Introduction and Results
%
%%%%%%%%%%%%%%%%%%%%%
%
%
%
\section{Introduction and Results}
We establish local well-posedness of the following reaction-diffusion equation with power nonlinearity formulated on the half-line with a nonzero Dirichlet boundary condition: 
\sss{\label{nh-rd-ibvp}
\ddd{
& u_t - u_{xx} = |u|^{p-1}u, \quad && x \in (0, \infty), \ t \in (0, T), 
\label{nh-rd-eq}
\\
&u(x,0) = u_0(x), && x \in [0, \infty),
\label{nh-ibvp-ic-nh}
\\
&u(0,t) = g_0(t), \quad && t \in [0, T],
\label{nh-ibvp-bc-nh}
}
}
where $p=2, 3, 4, \ldots,$ and $T<1$. 

The Sobolev spaces $H_x^s(0, \infty)$ for the initial datum $u_0$ and $H_t^{\frac{2s+1}{4}}(0, T)$ for the boundary datum $g_0$ of the above initial-boundary value problem (IBVP) are obtained as restrictions of their whole line counterparts according to the general definition
\eee{
H^s(\Omega)
=
\left\{
f: f=F\big|_{\Omega} \text{ with } F\in H^s(\mathbb R)
\right\},
\quad
\Omega\subset \mathbb R.
\nn
}
Furthermore, we shall see that the correspondence $s\leftrightarrow \frac{2s+1}{4}$ between the regularity (in $x$) of  $u_0$ given in \eqref{nh-ibvp-ic-nh} and the regularity (in $t$) of  $g_0$ specified in \eqref{nh-ibvp-bc-nh} is determined by the \textit{linear} part of  equation \eqref{nh-rd-eq} (that is, the linear heat equation) via two independent directions: (i) the space regularity of the linear version of IBVP \eqref{nh-rd-ibvp} with zero initial data, and (ii) the time regularity of the  linear heat initial value problem (IVP) with data in $H_x^s(\mathbb R)$. 
In the latter direction, especially remarkable is the fact that the time regularity of the solution of the linear heat IVP is described by the space $H_t^{\frac{2s+1}{4}}(0, T)$, which is the one also associated with the linear Schr\"odinger IVP (see, for example, \cite{kpv1991, fhm2017}). 
This is rather surprising, taking into account the  sharp contrast between the diffusive nature of the linear heat equation and the dispersive character of the linear Schr\"odinger equation.

The reaction-diffusion equation \eqref{nh-rd-eq} has been studied extensively and from various points of view, see, for example,  \cite{a1979, w1980,  gk1985, g1986b,bcg1998, egkp2004, mm2009, cdw2009b, g2010}, the books \cite{sm1994, h2018} and the references therein.
All of these works are concerned either with the IVP or with IBVPs formulated with a \textit{zero}  boundary condition. In the case of the half-line IBVP  \eqref{nh-rd-ibvp}, this corresponds to taking $g_0\equiv 0$. 
On the contrary, here we consider the case of \textit{nonzero} Dirichlet boundary conditions. 

In fact, the main objective of this work is to take advantage of a novel approach for the well-posedness of nonlinear evolution equations, which relies on the new solution formulae produced by the unified transform method of Fokas \cite{f1997, f2008} for the forced linear counterparts of these equations and which was originally developed for \textit{dispersive} equations such as the nonlinear Schr\"odinger \cite{fhm2017}, the Korteweg-de Vries \cite{fhm2016} and the ``good'' Boussinesq \cite{hm2015} equations. That is, the purpose of this work is to provide the \textit{extension} of the aforementioned new approach from dispersive equations to \textit{diffusive} ones.

For ``smooth'' data ($s>\frac 12$),  local well-posedness for the reaction-diffusion IBVP \eqref{nh-rd-ibvp} will be established in the following sense.
\begin{theorem}[\b{Well-posedness on the half-line for ``smooth'' data}]\label{nh-ibvp-wp-h-t}
Suppose $\frac 12 < s < \frac 32$, $u_0\in H_x^s(0, \infty)$ and 
$g_0\in H_t^{\frac{2s+1}{4}}(0, T)$ with the compatibility condition $u_0(0) = g_0(0)$. 
Furthermore, for  
\ddd{
&\left\|\left(u_0, g_0\right)\right\|_D
=
\no{u_0}_{H_x^s(0,\infty)}
+ 
\no{g_0}_{H_t^{\frac{2s+1}{4}}(0,T)},
\nn\\
&T^*
=
\min
\bigg\{ 
T,  \, 
\frac{1}{p^2 \left(2c_{s, p}\right)^{2p} \no{(u_0, g_0)}_D^{2\left(p-1\right)}}
\bigg\},
\quad c_{s, p}>0,
\nn
}
define  the space $X$ on the half-line by
\ddd{
&X = C\big([0,T^*]; H_x^s(0,\infty)\big)\cap C\big([0,\infty); H_t^{\frac{2s+1}{4}}(0, T^*)\big),
\nn\\
&
\no{u}_X 
= \sup_{t\in [0,T^*]} \no{u(t)}_{H_x^s(0,\infty)} + \sup_{x\in [0, \infty)} \no{u(x)}_{H_t^{\frac{2s+1}{4}}(0, T^*)}.
\nn
}
Then, for $\frac{p-1}{2}\in\mathbb N$ there exists a unique solution $u\in X$ to the reaction-diffusion IBVP \eqref{nh-rd-ibvp}, which satisfies the estimate
\eee{
%\label{nh-nh-ibvp-x-est} 
\no{u}_X
\leqslant 
2 c_{s, p} \left\|\left(u_0, g_0\right)\right\|_D.
\nn
}
Furthermore, the data-to-solution map
$\left\{u_0, g_0 \right\}\mapsto u $
 is locally Lipschitz continuous.
\end{theorem}

For ``rough'' data ($s<\frac 12$), we refine our solution space by intersecting $X$ with the space
\eee{\label{nh-cbeta-def}
\begin{split}
C^\alpha([0, T]; L_x^p(0, \infty))
=
\Big\{
u \in C([0, T]; L_x^p(0, \infty)): 
\sup_{t\in [0,T]} t^\alpha \no{u(t)}_{L_x^p(0,\infty)} < \infty
\Big\},
\\
\alpha = \tfrac{1}{p}\left(\tfrac 12 - b\right), 
\
\tfrac{2s+1}{4} < b < \tfrac 12.
\end{split}
}
In this case, our  result reads as follows.
\begin{theorem}[\b{Well-posedness on the half-line for ``rough'' data}]\label{nh-ibvp-wp-l-t}
Suppose  $\tfrac 12 - \tfrac{1}{p} < s<\tfrac 12$, $u_0\in H_x^s(0, \infty)$ and 
$g_0\in H_t^{\frac{2s+1}{4}}(0, T)$, and for $\left\|\left(u_0, g_0\right)\right\|_D$  as in Theorem \ref{nh-ibvp-wp-h-t} and lifespan $T^*$ given by
\eee{
T^*
=
\min\bigg\{T,\,   \frac{1}{\left(2^{p+2}p\right)^{\frac{1}{\alpha}} 
\left(2c_{s, p}\right)^{\frac{p}{\alpha}}  \no{(u_0, g_0)}_D^{\frac{p-1}{\alpha}}}
\bigg\}, \quad c_{s, p}>0,
\nn
}
define the space $Y$ on the half-line by
\ddd{
&Y
=
C\big([0,T^*]; H_x^s(0,\infty)\big)\cap C\big([0,\infty); H_t^{\frac{2s+1}{4}}(0, T^*)\big) \cap C^\alpha([0, T^*]; L_x^p(0, \infty)),
\nn\\
&\no{u}_{Y} 
=
\sup_{t\in [0,T^*]} \no{u(t)}_{H_x^s(0,\infty)} 
+ \sup_{x\in [0, \infty)} \no{u(x)}_{H_t^{\frac{2s+1}{4}}(0, T^*)}
+ \sup_{t\in [0,T^*]} t^\alpha \no{u(t)}_{L_x^p(0,\infty)}.
\nn
}
Then,  the reaction-diffusion IBVP \eqref{nh-rd-ibvp} has a unique solution $u\in Y$, which satisfies the estimate
\eee{
%\label{nh-nh-ibvp-y-est} 
\no{u}_Y
\leqslant 
2c_{s, p} \no{\left(u_0, g_0\right)}_D.
\nn
}
Furthermore, the data-to-solution map
$\left\{u_0, g_0 \right\}\mapsto u $
 is locally Lipschitz continuous.
\end{theorem}

Theorems \ref{nh-ibvp-wp-h-t} and \ref{nh-ibvp-wp-l-t} will be established by showing that the iteration map induced by an appropriate solution representation of the forced linear counterpart of the nonlinear problem is a contraction in a suitably chosen solution space.
In the case of the nonlinear IBVP \eqref{nh-rd-ibvp}, the associated linear problem is the forced linear heat IBVP:
\sss{\label{nh-flh-ibvp}
\ddd{
& u_t - u_{xx} = f(x, t),  \quad && x \in (0, \infty), \ t \in (0, T),   
\label{nh-flh-ibvp-eq}
\\
&u(x,0) = u_0(x),  && x \in [0, \infty),
\label{nh-flh-ibvp-ic}
\\ 
&u(0,t) = g_0(t),  && t \in [0, T].
\label{nh-flh-ibvp-bc}
}
}

Recall that we consider nonzero boundary conditions; hence, the linear IBVP \eqref{nh-flh-ibvp} cannot be converted into an IVP via the reflection method. Moreover, the classical sine transform solution formula for this problem is not convenient for the purpose of estimates due to its oscillatory nature. Thus, a significant obstacle is present already at the very beginning of the contraction mapping argument, namely at the stage of simply specifying the iteration map. This stands in stark contrast with the case of the IVP, for which the linear problem is solved by means of a straightforward application of the Fourier transform.

A novel approach was recently introduced for the well-posedness of IBVPs involving nonlinear evolution equations. This approach bypasses the absence of  Fourier transform in the IBVP setting by utilizing the  unified transform method (UTM) of Fokas  for the explicit solution of forced linear evolution IBVPs \cite{f1997, f2008}. The new approach has already been implemented for the nonlinear Schr\"odinger, the Korteweg-de Vries and the ``good'' Boussinesq equations on the half-line \cite{fhm2017, fhm2016, hm2015}.
These three IBVPs share two things in common: (i) they concern dispersive equations, and (ii) they are formulated on the half-line. 
The present work extends the UTM approach to IBVP well-posedness in two different directions: (i) from \textit{dispersive} to \textit{diffusive} equations, and (ii) from the \textit{half-line} to the  \textit{finite interval}.

In the case of the forced linear IBVP \eqref{nh-flh-ibvp}, UTM yields the solution formula  \eqref{nh-flh-utm-sol-T}. Theorems \ref{nh-ibvp-wp-h-t} and \ref{nh-ibvp-wp-l-t} will be established by deriving the linear estimates necessary for showing that the iteration map obtained by setting $f=|u|^{p-1}u$ in  \eqref{nh-flh-utm-sol-T} has a unique fixed point in the relevant solution spaces.
In fact, the analysis of the half-line problem \eqref{nh-flh-ibvp} can also be exploited to show well-posedness of the following reaction-diffusion IBVP on the finite interval:
\sss{\label{nh-fi-ibvp}
\ddd{
& u_t - u_{xx} = |u|^{p-1}u,   && x \in (0, \ell), \ t \in (0, T),
\label{nh-fi-ibvp-eq-nh}
\\
&u(x,0) = u_0(x)\in H_x^s(0, \ell),  && x \in [0, \ell],
\label{nh-fi-ibvp-ic-nh}
\\[-0.25em]
&u(0,t) = g_0(t)\in H_t^{\frac{2s+1}{4}}(0, T), && t \in [0, T],
\\[-0.25em]
&u(\ell,t) = h_0(t)\in H_t^{\frac{2s+1}{4}}(0, T), \quad && t \in [0, T],
\label{nh-fi-ibvp-bc-nh}
}
}
where $p=2, 3, 4, \ldots,$ and $T<1$. In particular, we shall prove the following result.

\begin{theorem}[\b{Well-posedness on a finite interval  for ``smooth'' data}]\label{nh-fi-ibvp-wp-h-t}
Suppose $\frac 12 < s < \frac 32$,
$u_0\in H_x^s(0, \ell)$, 
$g_0\in H_t^{\frac{2s+1}{4}}(0, T)$ and $h_0\in H_t^{\frac{2s+1}{4}}(0, T)$ with the compatibility conditions $u_0(0) = g_0(0)$ and $u_0(\ell) = h_0(0)$. Furthermore, for
\ddd{
&\left\|\left(u_0, g_0, h_0\right)\right\|_D
=
\no{u_0}_{H_x^s(0,\ell)}
+ 
\no{g_0}_{H_t^{\frac{2s+1}{4}}(0,T)}
+ 
\no{h_0}_{H_t^{\frac{2s+1}{4}}(0,T)},
\nn\\
&T^*
=\min
\bigg\{ 
T,  \, 
\frac{1}{p^2\left(2c_{s, p}\right)^{2p} \no{(u_0, g_0, h_0)}_D^{2\left(p-1\right)}}
\bigg\},
\quad
c_{s, p}>0,
\nn
}
define the space $X$ on the interval by
\ddd{
&X 
= C\big([0,T^*]; H_x^s(0,\ell)\big)\cap C\big([0,\ell]; H_t^{\frac{2s+1}{4}}(0, T^*)\big),
\nn\\
&\no{u}_X 
= \sup_{t\in [0,T^*]} \no{u(t)}_{H_x^s(0,\ell)} + \sup_{x\in [0, \ell]} \no{u(x)}_{H_t^{\frac{2s+1}{4}}(0, T^*)}.
\nn
}
Then, for $\frac{p-1}{2}\in\mathbb N$ there exists a unique solution $u \in X$ to the reaction-diffusion IBVP \eqref{nh-fi-ibvp}, which satisfies the  estimate
\eee{
%\label{nh-fi-ibvp-x-est} 
\no{u}_X
\leqslant 
2c_{s, p} \left\|\left(u_0, g_0, h_0\right)\right\|_D.
\nn
}
Furthermore, the data-to-solution map
$\left\{u_0, g_0, h_0 \right\}\mapsto u $
 is locally Lipschitz continuous.
\end{theorem}

\vskip 3mm

For ``rough'' data ($s<\frac 12$), in analogy to the half-line we refine our solution space by additionally including the space
\eee{\label{nh-fi-cbeta-def}
C^\alpha([0, T]; L_x^p(0, \ell))
=
\Big\{
u \in C([0, T]; L_x^p(0, \ell)): 
\sup_{t\in [0,T]} t^\alpha \no{u(t)}_{L_x^p(0,\ell)} < \infty
\Big\}
}
with $\alpha$  as in \eqref{nh-cbeta-def}.
In this case, we have the following result.
\begin{theorem}[\b{Well-posedness on the finite interval  for ``rough'' data}]\label{nh-fi-ibvp-wp-l-t}
Suppose  $\tfrac 12 - \tfrac{1}{p} < s<\tfrac 12$, $u_0\in H_x^s(0, \ell)$, 
$g_0\in H_t^{\frac{2s+1}{4}}(0, T)$ and $h_0\in H_t^{\frac{2s+1}{4}}(0, T)$, and for $\left\|\left(u_0, g_0, h_0\right)\right\|_D$  as in Theorem \ref{nh-fi-ibvp-wp-h-t} and   lifespan $T^*$  given by
\eee{
T^*
=
\min\bigg\{T,\,   \frac{1}{\left(2^{p+2}p\right)^{\frac{1}{\alpha}}
\left(2c_{s, p}\right)^{\frac{p}{\alpha}} \no{(u_0, g_0, h_0)}_D^{\frac{p-1}{\alpha}}}
\bigg\}, \quad c_{s, p}>0,
\nn
}
define the space $Y$ on the interval by
\ddd{
&Y
=
C\big([0,T^*]; H_x^s(0,\ell)\big)\cap C\big([0,\ell]; H_t^{\frac{2s+1}{4}}(0, T^*)\big) \cap C^\alpha([0, T^*]; L_x^p(0, \ell))
\nn\\
&\no{u}_{Y} 
= 
\sup_{t\in [0,T^*]} \no{u(t)}_{H_x^s(0,\ell)} 
+ \sup_{x\in [0, \ell]} \no{u(x)}_{H_t^{\frac{2s+1}{4}}(0, T^*)}
+ \sup_{t\in [0,T^*]} t^\alpha \no{u(t)}_{L_x^p(0,\ell)}.
}
Then,  the reaction-diffusion IBVP \eqref{nh-fi-ibvp} has a unique solution  $u\in Y$, which satisfies the estimate
\eee{
%\label{nh-fi-ibvp-y-est} 
\no{u}_Y
\leqslant 
2 c_{s, p} \no{\left(u_0, g_0, h_0\right)}_D.
\nn
}
Furthermore, the data-to-solution map
$\left\{u_0, g_0, h_0 \right\}\mapsto u $
 is locally Lipschitz continuous.
\end{theorem}

Theorems \ref{nh-fi-ibvp-wp-h-t} and \ref{nh-fi-ibvp-wp-l-t} for the reaction-diffusion IBVP \eqref{nh-fi-ibvp} on the finite interval will be established by deriving appropriate estimates for the forced linear IBVP
\sss{\label{nh-fi-flh-ibvp}
\ddd{
& u_t - u_{xx} = f(x, t),  \quad && x \in (0, \ell), \ t \in (0, T), 
\label{nh-fi-flh-ibvp-eq-nh}
\\
&u(x,0) = u_0(x),  && x \in [0, \ell],
\label{nh-fi-flh-ibvp-ic-nh}
\\ 
&u(0,t) = g_0(t), && t \in [0, T],
\\ 
&u(\ell,t) = h_0(t),  && t \in [0, T].
\label{nh-fi-flh-ibvp-bc-nh}
}
}
Actually, thanks to the linear estimates obtained for the half-line problem \eqref{nh-flh-ibvp}, it will be sufficient to estimate the finite interval problem \eqref{nh-fi-flh-ibvp} in the reduced case $f = u_0 = g_0 =0$.

We note that apart from the UTM approach to IBVP well-posedness there are also other approaches in the literature, namely the works of Colliander, Kenig and Holmer \cite{ck2002, h2005, h2006} as well as of Bona, Sun and Zhang \cite{bsz2002} for the Korteweg-de Vries and the nonlinear Schr\"odinger equations on the half-line.
Furthermore, for a different treatment of linear and nonlinear evolution IBVPs that combines UTM with inverse scattering techniques we refer the reader to Fokas and Pelloni \cite{fp2000}, Pelloni \cite{p2002, p2004},  Fokas, Its and Sung \cite{fis2005}, Fokas and Lenells \cite{fl2012, lf2012a, lf2012b}, Fokas and Spence \cite{fs2012}, Deconinck, Pelloni and  Sheils \cite{dps2014}, Sheils and Smith \cite{ss2015}, and the references therein.

Finally, we would like to emphasize that one could possibly study the IBVPs 
considered in this work via classical techniques developed specifically for diffusive equations with data in Lipschitz  spaces (see, for example, \cite{a1979, w1980, g1986b, sm1994}), although we were not able to find the results obtained in this work in the classical literature.
However, as mentioned earlier, our objective here is different, namely to demonstrate that the reaction-diffusion equation \eqref{nh-rd-eq} can be included in the theory  originally developed for dispersive equations like NLS \cite{fhm2017}, KdV \cite{fhm2016} and ``good'' Boussinesq \cite{hm2015}, which relies on employing UTM and the $L^2$-boundedness of the Laplace transform (see Lemma \ref{nh-tlap-l}) for deriving novel linear estimates along rays in the \textit{complex} Fourier space.

\vskip 3mm
\noindent
\textbf{Structure of the paper.}
The UTM solution formulae for the forced linear heat equation on the half-line and on the finite interval are derived in Section \ref{nh-utm-s}. Estimates for the linear heat IVP are obtained in Section \ref{nh-ivp-s}, including those that motivate the refined solution spaces of Theorems \ref{nh-ibvp-wp-l-t} and \ref{nh-fi-ibvp-wp-l-t}. The corresponding estimates for the forced linear heat equation on the half-line are derived in Section \ref{nh-ibvp-le-s} and are subsequently employed in Section \ref{nh-lwp-s} for carrying out the contraction mapping argument leading to Theorems \ref{nh-ibvp-wp-h-t} and \ref{nh-ibvp-wp-l-t}.
Finally, the forced linear heat equation on the finite interval is estimated  in Section \ref{nh-fi-ibvp-le-s} and the resulting estimates yield Theorems \ref{nh-fi-ibvp-wp-h-t} and \ref{nh-fi-ibvp-wp-l-t} via a contraction mapping argument analogous to that of Section \ref{nh-lwp-s}.
%
%
%
%%%%%%%%%%%%%%%%%%%%%%% 
%
%	Linear Solution Formulae via UTM
%
%%%%%%%%%%%%%%%%%%%%%%%
%
%
%
\section{Linear Solution Formulae via UTM}
\label{nh-utm-s}

In this section, we employ UTM in order to derive the solution formulae  \eqref{nh-flh-utm-sol-T} and  \eqref{nh-fi-flh-utm-sol-T} for the forced linear IBVPs \eqref{nh-rd-ibvp} and \eqref{nh-fi-ibvp} on the half-line and on the finite interval respectively. The derivations presented here can be found in more detail in \cite{f2008}.

\subsection{The half-line}

Let $\widetilde u$ satisfy the adjoint equation of the linear heat    equation, i.e.
%
%\eee{\label{nh-lh-adj}
$\widetilde u_t+\widetilde u_{xx} = 0$.
Multiplying this equation by $u$ and adding it to the forced linear  heat equation \eqref{nh-flh-ibvp-eq} multiplied by $\tilde u$, we arrive at the divergence form
$
\left(\widetilde u u\right)_t-\left(\widetilde u u_x-\widetilde u _xu\right)_x
=
\widetilde u  f. 
$
Setting $\widetilde u(x, t)=e^{-ikx+k^2t}$ and integrating with respect to  $x$ and $t$ yields the \textit{global relation}
\eee{\label{nh-gr}
e^{k^2t}\, \widehat u (k,t)
=
\widehat u_0(k)
-
\left[\widetilde g_1(k^2, t) + ik \widetilde g_0(k^2, t) \right]
+
\int_{t'=0}^t e^{k^2t'} \widehat f (k,t') dt',
\quad \text{Im}(k) \leqslant 0,
}
where 
\eee{\nn
\widehat u (k,t)
=
\int_{x=0}^\infty  e^{-ikx} u(x,t)dx,
\quad 
\widetilde g_j(k^2, t) = \int_{t'=0}^t e^{k^2t'} \p_x^j u(0, t')dt', \ j=0,1,
}
and the half-line Fourier transforms $\what u_0$, $\what f$ are defined analogously to $\what u$. Note that these transforms are analytic for $\text{Im}(k)<0$ thanks to the boundedness of $e^{ikx}$ for all $x>0$ and a Paley-Wiener-type theorem (see, for example, Theorem 7.2.4 in \cite{s1994}).
Inverting \eqref{nh-gr}, we obtain  
\ddd{
\label{nh-ir}
u(x,t)
&=
\frac{1}{2\pi}\int_{k\in\mathbb R} e^{ikx-k^2t}\, \widehat u_0(k) dk
+
\frac{1}{2\pi}\int_{k\in\mathbb R} e^{ikx-k^2t} \int_{t'=0}^t e^{k^2 t'} \widehat f(k, t') dt' dk
\nn\\
&\quad
-
\frac{1}{2\pi}\int_{k\in\mathbb R} e^{ikx-k^2t} \left[\widetilde g_1(k^2,t)+ik\widetilde g_0(k^2,t) \right]dk.
}

The integral representation \eqref{nh-ir} involves the unknown Neumann boundary value $u_x(0, t)$ through the term $\widetilde g_1$. However, it turns out  that this term can be eliminated.
Indeed, since $x\geqslant 0$ and $t\geqslant t'$, the exponential $e^{ikx-k^2(t-t')}$ is bounded whenever $0\leqslant \text{Im}(k) \leqslant \left|\text{Re}(k)\right|$. Thus, thanks to the analyticity and exponential decay  in $k$, Cauchy's theorem implies
\ddd{\label{nh-ir-def}
u(x,t)
&=
\frac{1}{2\pi}\int_{k\in\mathbb R} e^{ikx-k^2t} \, \widehat u_0(k) dk
+
\frac{1}{2\pi}\int_{k\in\mathbb R} e^{ikx-k^2t} \int_{t'=0}^t e^{ik^2 t'} \widehat f(k, t') dt' dk
\nn\\
&\quad
-
\frac{1}{2\pi}\int_{k\in \p D^+} e^{ikx-k^2t} \left[\widetilde g_1(k^2,t)+ik\widetilde g_0(k^2,t) \right]dk,
}
where  $\p D^+$ is the positively oriented boundary of the region $D^+=
\left\{k\in\mathbb C: \text{Im}(k) \geqslant \left|\text{Re}(k)\right|\right\}$, as shown in Figure \ref{nh-fi-flh-dpm}.

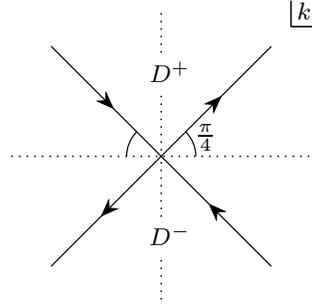
\begin{figure}[ht]
\centering
\vspace{1.3cm}
\begin{tikzpicture}[scale=1.15]
\pgflowlevelsynccm
\draw[line width=.5pt, black, dotted](1.8,0)--(-1.8,0);
\draw[line width=.5pt, black, dotted](0,0.7)--(0,0);
\draw[line width=.5pt, black, dotted](0,1.2)--(0,1.7);
\draw[line width=.5pt, black](1.5,1.8)--(1.5,1.5);
\draw[line width=.5pt, black](1.5,1.5)--(1.8,1.5);
\draw[line width=.5pt, black, dotted](0, -0.7)--(0,0);
\draw[line width=.5pt, black, dotted](0, -1.2)--(0, -1.7);
\node[] at (1.59, 1.685) {\fontsize{8}{8} $k$};
\draw[middlearrow={Stealth[scale=1.3, reversed]}] (0,0) -- (135:1.8);
\draw[middlearrow={Stealth[scale=1.3]}] (0,0) -- (45:1.8);
\draw[middlearrow={Stealth[scale=1.3]}] (0,0) -- (225:1.8);
\draw[middlearrow={Stealth[scale=1.3, reversed]}] (0,0) -- (315:1.8);
\node[] at (0.1, 0.98) {\fontsize{8}{10}\it $D^+$};
\node[] at (0.1, -0.91) {\fontsize{8}{10}\it $D^-$};
\draw (4mm,0) arc (0:45:4mm);
\draw (-4mm,0) arc (180:135:4mm);
\node[] at (0.45, 0.22) {\fontsize{8}{8} $\frac \pi 4$};
\end{tikzpicture}
\vspace{18mm}
\caption{The regions $D^\pm$ and their positively oriented boundaries $\p D^\pm$.}
\label{nh-fi-flh-dpm}
\end{figure}

Moreover, the transformation $k\mapsto -k$ applied to the global relation \eqref{nh-gr} implies the identity
\eee{\nn
e^{k^2t}\, \widehat u (-k,t)
=
\widehat u_0(-k)
- \left[\widetilde g_1(k^2, t) - ik \widetilde g_0(k^2, t) \right]
+\int_{t'=0}^t e^{k^2t'} \widehat f (-k,t') dt',
\quad \text{Im}(k) \geqslant 0.
}
Solving for   $\widetilde g_1$ and inserting the resulting expression in \eqref{nh-ir-def} yields the solution of IBVP \eqref{nh-flh-ibvp} as
\ddd{
\label{nh-flh-utm-sol-t}
u(x,t)
&=
\frac{1}{2\pi}\int_{k\in\mathbb R} e^{ikx-k^2t}\, \widehat u_0(k) dk
-
\frac{1}{2\pi}\int_{k\in\p D^+} e^{ikx-k^2t}\, \widehat u_0(-k) dk
\nn\\
&\quad
+
\frac{1}{2\pi}\int_{k\in\mathbb R} e^{ikx-k^2t} \int_{t'=0}^t e^{k^2 t'} \widehat f(k, t') dt' dk
-
\frac{1}{2\pi}\int_{k\in \p D^+} e^{ikx-k^2t} \int_{t'=0}^t e^{k^2 t'} \widehat f(-k, t') dt' dk
\nn\\
&\quad
-
\frac{i}{\pi}\int_{k\in \p D^+} e^{ikx-k^2t} \, k\widetilde g_0(k^2,t) dk,
}
where we have made crucial use of the fact that
$
\int_{k\in \p D^+}e^{ikx} \, \widehat u(-k,t)dk = 0
$
due to analyticity and exponential decay of the integrand  in  $D^+$.
Finally, exploiting once again analyticity and exponential decay in $D^+$, we observe that
$
\int_{k\in \p D^+} e^{ikx-k^2t} \int_{t'=t}^T e^{k^2 t'} u(0, t') dt' dk = 0.
$
Hence, the solution formula \eqref{nh-flh-utm-sol-t} can be written in the following convenient form for deriving linear estimates:
\ddd{\label{nh-flh-utm-sol-T}
u(x,t)
&=
S\big[u_0, g_0; f\big](x, t)
\nn\\
&=
\frac{1}{2\pi}\int_{k\in\mathbb R} e^{ikx-k^2t}\, \widehat u_0(k) dk
-
\frac{1}{2\pi}\int_{k\in\p D^+} e^{ikx-k^2t}\, \widehat u_0(-k) dk
\nn\\
&\quad
+
\frac{1}{2\pi}\int_{k\in\mathbb R} e^{ikx-k^2t} \int_{t'=0}^t e^{k^2 t'} \widehat f(k, t') dt' dk
-
\frac{1}{2\pi}\int_{k\in \p D^+} e^{ikx-k^2t} \int_{t'=0}^t e^{k^2 t'} \widehat f(-k, t') dt' dk
\nn\\
&\quad
-
\frac{i}{\pi}\int_{k\in \p D^+} e^{ikx-k^2t} \, k\widetilde g_0(k^2, T) dk,
\qquad
\widetilde g_0(k^2, T) \doteq \int_{t=0}^T e^{k^2 t} g_0(t) dt.
}

%
%
%
%%%%%%%%%%%%%%%  
%
%	The finite interval
%
%%%%%%%%%%%%%%%
%
%
%
\subsection{The finite interval}
Integrating the divergence form 
$
\left(\widetilde u u\right)_t-\left(\widetilde u u_x-\widetilde u _xu\right)_x
=
\widetilde u  f
$
with respect to $x$ and $t$, we obtain the  global relation
\ddd{\label{nh-fi-gr}
e^{k^2t} \widehat u(k, t)
=
\widehat u_0(k)
+
e^{-ik\ell} \left[\widetilde h_1(k^2, t) + ik \widetilde h_0(k^2, t)\right]
&- 
\left[\widetilde g_1(k^2, t) + ik \widetilde g_0(k^2, t) \right] 
\nn\\
&+
\int_{t'=0}^t 
e^{k^2 t'} \widehat f(k, t') dt', 
\quad k\in\mathbb C,
}
where $\widehat u(k, t) = \int_{x=0}^\ell e^{-ikx} u(x, t) dx$ with $\what u_0$ and $\what f$ defined similarly, and 
\eee{\nn
\widetilde g_j(k^2, t) = \int_{t'=0}^t e^{k^2 t'} \p_x^ju(0, t')dt',
\quad
\widetilde h_j(k^2, t) = \int_{t'=0}^t e^{k^2 t'} \p_x^j u(\ell, t')dt',
\quad j=0, 1.
}
Note that the above spatial transforms make sense for all $k\in\mathbb C$ since $x\in [0, \ell]$.
Inverting \eqref{nh-fi-gr} gives rise to the integral representation
\ddd{\label{nh-fi-ir}
u(x, t)
&=
\frac{1}{2\pi}\int_{k\in\mathbb R} e^{ikx-k^2t} \widehat u_0(k) dk
+
\frac{1}{2\pi}\int_{k\in\mathbb R} e^{ikx-k^2t} \int_{t'=0}^t e^{k^2t'} \widehat f(k, t') dt' dk
\nn\\
&\quad
-
\frac{1}{2\pi}\int_{k\in\mathbb R} e^{ikx-k^2 t}
\left[\widetilde g_1(k^2, t)+ik \widetilde g_0(k^2, t)\right]dk
\nn\\
&\quad
+
\frac{1}{2\pi}\int_{k\in\mathbb R} e^{ik(x-\ell)-k^2 t}
\left[\widetilde h_1(k^2, t)+ik \widetilde h_0(k^2, t)\right]dk.
}
The unknown Neumann values $u_x(0, t)$ and $u_x(\ell, t)$, which are contained in $\widetilde g_1$ and $\widetilde h_1$, can be eliminated from  the above representation similarly to the half-line case. In particular, observe that the exponentials $e^{ikx-k^2(t-t')}$ and $e^{ik(x-\ell)-k^2(t-t')}$  are bounded in  $\left\{\text{Im}(k)\geqslant 0\right\} \setminus D^+$  and  $\left\{\text{Im}(k)\leqslant 0\right\} \setminus D^-$ respectively, where $D^+$ is defined as in the half-line and $D^- = \left\{k\in\mathbb C:\text{Im}(k) < -\left|\text{Re}(k)\right|\right\}$.
Hence, using Cauchy's theorem we can deform the contours of integration in the second and third integrals of \eqref{nh-fi-ir} to the positively oriented boundaries $\p D^\pm$  of $D^\pm$  to obtain
\ddd{\label{nh-fi-ir-def}
u(x, t)
&=
\frac{1}{2\pi}\int_{k\in\mathbb R} e^{ikx-k^2t} \widehat u_0(k) dk
+
\frac{1}{2\pi}\int_{k\in\mathbb R} e^{ikx-k^2t} \int_{t'=0}^t e^{k^2t'} \widehat f(k, t') dt' dk
\nn\\
&\quad
-
\frac{1}{2\pi}\int_{k\in\p D^+} e^{ikx-k^2 t}
\left[\widetilde g_1(k^2, t)+ik \widetilde g_0(k^2, t)\right]dk
\nn\\
&\quad
-
\frac{1}{2\pi}\int_{k\in\p D^-} e^{ik(x-\ell)-k^2 t}
\left[\widetilde h_1(k^2, t)+ik \widetilde h_0(k^2, t)\right]dk.
}
Under the transformation $k\mapsto -k$, the global relation \eqref{nh-fi-gr} yields the additional identity
\ddd{\label{nh-fi-gr-sym}
e^{k^2t} \widehat u(-k, t)
=
\widehat u_0(-k)
+
e^{ik\ell} \left[\widetilde h_1(k^2, t)- ik \widetilde h_0(k^2, t)\right]
&- 
\left[\widetilde g_1(k^2, t) - ik \widetilde g_0(k^2, t) \right] 
\nn\\
&+
\int_{t'=0}^t 
e^{k^2 t'} \widehat f(-k, t') dt', 
\quad k\in\mathbb C.
}
Solving \eqref{nh-fi-gr} and \eqref{nh-fi-gr-sym} for $\widetilde g_1$  and $\widetilde h_1$, and noting that the terms $\widehat u(k, t)$ and $\widehat u(-k, t)$ involved in the resulting expressions yield zero contribution when inserted in \eqref{nh-fi-ir-def} thanks to analyticity and exponential decay, we obtain the  solution of IBVP \eqref{nh-fi-flh-ibvp} in the form
\ddd{\label{nh-fi-flh-utm-sol-T}
u(x, t)
&=
S\big[u_0, g_0, h_0; f\big](x, t)
\nn\\
&=
\frac{1}{2\pi}\int_{k\in\mathbb R} e^{ikx-k^2t} \widehat u_0(k) dk
-
\frac{1}{2\pi}\int_{k\in\p D^+} 
\frac{e^{ikx-k^2 t}}{e^{ik\ell}-e^{-ik\ell}}
\left[e^{ik\ell}\widehat u_0(k)-e^{-ik\ell} \widehat u_0(-k)\right]dk
\nn\\
&\quad
-
\frac{1}{2\pi}\int_{k\in\p D^-} 
\frac{e^{ik(x-\ell)-k^2 t}}{e^{ik\ell}-e^{-ik\ell}}
\left[\widehat u_0(k)- \widehat u_0(-k)\right]dk
%\nn\\
%&\quad
+
\frac{1}{2\pi}\int_{k\in\mathbb R} e^{ikx-k^2t} \int_{t'=0}^t e^{k^2t'} \widehat f(k, t') dt' dk
\nn\\
&\quad
-
\frac{1}{2\pi}\int_{k\in\p D^+} 
\frac{e^{ikx-k^2 t}}{e^{ik\ell}-e^{-ik\ell}}
\left[e^{ik\ell} \int_{t'=0}^t e^{k^2t'} \widehat f(k, t') dt' -e^{-ik\ell} \int_{t'=0}^t  e^{k^2t'} \widehat f(-k, t') dt'\right]dk
\nn\\
&\quad
-
\frac{1}{2\pi}\int_{k\in\p D^-} 
\frac{e^{ik(x-\ell)-k^2 t}}{e^{ik\ell}-e^{-ik\ell}}
\left[\int_{t'=0}^t e^{k^2t'} \widehat f(k, t') dt'-\int_{t'=0}^t e^{k^2t'} \widehat f(-k, t') dt'\right]dk
\nn\\
&\quad
-
\frac{1}{2\pi}\int_{k\in\p D^+} 
\frac{e^{ikx-k^2 t}}{e^{ik\ell}-e^{-ik\ell}}
\left[-2ik e^{-ik\ell} \widetilde g_0(k^2, T)+2ik \widetilde h_0(k^2,T)\right]dk
\nn\\
&\quad
-
\frac{1}{2\pi}\int_{k\in\p D^-} 
\frac{e^{ik(x-\ell)-k^2 t}}{e^{ik\ell}-e^{-ik\ell}}
\left[-2ik  \widetilde g_0(k^2, T)+2ik e^{ik\ell}\, \widetilde h_0(k^2, T)\right]dk,
}
where, similarly to the half-line, we have  used analyticity and exponential decay in $D^\pm$ to replace the transforms $\widetilde g_0(k^2, t)$ and $\widetilde h_0(k^2, t)$ by
\eee{
\widetilde g_0(k^2, T) = \int_{t=0}^T e^{k^2 t} g_0(t) dt,
\quad
\widetilde h_0(k^2, T) = \int_{t=0}^T e^{k^2 t} h_0(t) dt.
\nn
}
%

%
%
%
%
%%%%%%%%%%%%%%%%
%
%	Linear IVP Estimates
%
%%%%%%%%%%%%%%%%
%
%
%
%
\section{Linear IVP Estimates}
\label{nh-ivp-s}

In this section, we write the forced linear IBVP \eqref{nh-rd-ibvp} as a combination of simpler IVPs and IBVPs that involve only one piece of data at a time and hence are easier to handle than the full problem. Then, we proceed to the estimation of the IVP component problems.

\subsection{Decomposition into simpler problems}
\label{nh-dec-ss}
Let  $U_0\in H_x^s(\mathbb R)$ be an extension of the initial datum $u_0\in H_x^s(0, \infty)$ such that
\eee{\label{nh-uU}
\no{U_0}_{H_x^s(\mathbb R)} \leqslant c\no{u_0}_{H_x^s(0, \infty)}, \quad s\geqslant 0, \ c\geqslant 1. 
}
Furthermore, since we are working with power nonlinearities, let the forcing of \eqref{nh-flh-ibvp} be of the form
$f(x, t) = \prod_{j=1}^p f_j(x, t)$ with  the whole line extension  
$F(x, t) = \prod_{j=1}^p F_j(x, t)$ defined as follows:
\vskip 2mm
\noindent
(i) For $s>\frac 12$, $F_j \in C\left(\left[0, T\right]; H_x^s(\mathbb R)\right)$ are extensions of $f_j\in C\big([0, T]; H_x^s(0,\infty)\big)$ such that
$$
\sup_{t\in [0, T]} \no{F_j(t)}_{H_x^s(\mathbb R)}
\leqslant
c \sup_{t\in [0, T]}  \no{f_j(t)}_{H_x^s(0, \infty)}, \quad c\geqslant 1, \ 1\leqslant j \leqslant p.
$$
Then, using the algebra property in $H_x^s(\mathbb R)$ we have
\eee{\label{nh-Ff-se}
\sup_{t\in [0, T]} \no{F(t)}_{H_x^s(\mathbb R)}
\leqslant
c_{s, p}  \prod_{j=1}^p \sup_{t\in [0, T]}   \no{f_j(t)}_{H_x^s(0, \infty)}.
}
\noindent
(ii) For $s\leqslant \frac 12$,  the functions $F_j$ are the extensions of the functions $f_j$ by zero outside $(0, \infty)$ so that by H\"older's inequality 
%$\big\|\prod_{j=1}^p f_j \big\|_{L^1}
%\leqslant
%\prod_{j=1}^p  \left\|f_j\right\|_{L^{r_j}}$,
%%
%$\sum_{j=1}^p \frac{1}{r_j} = 1$,
%
\eee{\label{nh-Ff-ca-est}
\no{F}_{C^{p\alpha}([0, T]; L_x^1(\mathbb R))}
\leqslant
\prod_{j=1}^p \no{f_j}_{C^\alpha([0, T]; L_x^p(0, \infty))}.
}
For above-defined $U_0$ and $F$, we write IBVP  \eqref{nh-flh-ibvp} as the superposition of the following problems:
\vskip 2mm
\noindent
\textbf{I.} The homogeneous linear IVP
\sss{\label{nh-lh-ivp}
\ddd{
&U _t -  U _{xx}=0,  && x\in \mathbb R,\ t\in (0, T),
\label{nh-lh-ivp-eq}\\
&U (x,0)= U_0(x)\in H_x^s(\mathbb R), \quad && x\in \mathbb R, \label{nh-lh-ivp-ic}
}
}
which can be solved via the whole line Fourier transform $\widehat U(\xi, t) = \int_{x\in \mathbb R} e^{-i\xi x}\, U(x, t) dx$ to yield 
\eee{
\label{nh-lh-ivp-sol}
U (x,t) = S\big[U_0; 0\big] (x, t) = \frac{1}{2\pi} \int_{\xi\in \mathbb R} e^{i\xi x-\xi^2t}\, \widehat U_0(\xi) d\xi.
}
\noindent
\textbf{II.} The forced linear IVP with zero initial condition
\sss{\label{nh-flh-ivp}
\ddd{
&W_t -  W_{xx}= F(x, t),\quad  && x\in \mathbb R,\ t\in (0, T),\label{nh-flh-ivp-eq}\\
& W(x,0)=0, && x\in \mathbb R, \label{nh-flh-ivp-ic}
}
}
whose solution is found via  the whole line Fourier transform $\widehat W(\xi, t) = \int_{x\in \mathbb R} e^{-i\xi x}\, W(x, t) dx$ as
\sss{\label{nh-flh-ivp-sol-comb}
\ddd{
W(x,t) 
= 
S\big[0; F\big](x, t)
&=
\frac{1}{2\pi} \int_{\xi\in \mathbb R} \int_{t'=0}^t  e^{i\xi x-\xi^2(t-t')}
\widehat F(\xi, t') dt' d\xi
\label{nh-flh-ivp-sol}
\\
&=
\int_{t'=0}^t  S\big[F(\cdot, t'); 0\big](x, t-t') dt'.
\label{nh-flh-ivp-sol-d}
}
}
\noindent
\textbf{III.} The homogeneous linear IBVP with zero initial condition
\sss{\label{nh-lh-ibvp-r}
\ddd{
&u_t-u_{xx} = 0, &&x\in (0,\infty), \ t\in (0,T),\label{nh-lh-ibvp-r-eq} \\
&u(x,0)= 0, &&x\in [0,\infty),  \label{nh-lh-ibvp-r-ic} \\
&u(0,t) = g_0(t) - U(0,t) \doteq G_0(t), \quad &&t\in [0, T], \label{nh-lh-ibvp-r-bc}
}
}
with solution $u=S\big[0, G_0; 0\big]$ given by the UTM formula \eqref{nh-flh-utm-sol-T} after setting $u_0=f=0$.
\vskip 2mm
\noindent
\textbf{IV.}  The homogeneous linear IBVP with  zero initial condition
\sss{\label{nh-flh-ibvp-r}
\ddd{
&u_t - u_{xx} = 0, &&x\in (0,\infty), \ t\in (0,T),
\label{nh-nls-hl-ls-fibvp} \\
&u(x,0)= 0, &&x\in [0,\infty),  \label{nh-nls-hl-fibvp-ic} \\
&u(0,t) = W(0, t), \quad &&t\in [0, T], 
\label{nh-nls-hl-w0-def}
}
}
with solution $u=S\big[0, W|_{x=0}; 0\big]$ given by the UTM formula \eqref{nh-flh-utm-sol-T} after setting $u_0=f=0$.

In summary, the UTM solution  \eqref{nh-flh-utm-sol-T} of the forced linear IBVP \eqref{nh-flh-ibvp} has been expressed as
\eee{\label{nh-flh-ibvp-sup}
S\big[u_0, g_0; f\big] 
=
S\big[U_0; 0\big]\Big|_{x\in (0, \infty)}
+
S\big[0; F\big]\Big|_{x\in (0, \infty)}
+
S\big[0, G_0; 0\big]
-
 S\big[0,  W|_{x=0}; 0\big],
}
where the four quantities on the right-hand side are the solutions of problems \eqref{nh-lh-ivp}, \eqref{nh-flh-ivp}, \eqref{nh-lh-ibvp-r} and \eqref{nh-flh-ibvp-r} respectively. 
In the remaining of the current section, we shall obtain estimates for the first two components of the superposition \eqref{nh-flh-ibvp-sup}, namely for the linear IVPs \eqref{nh-lh-ivp} and \eqref{nh-flh-ivp}. Then, in Section \ref{nh-ibvp-le-s} we shall derive corresponding estimates for the linear IBVPs \eqref{nh-lh-ibvp-r} and \eqref{nh-flh-ibvp-r}.
%
%
%
%
%%%%%%%%%%%%%%%%%%%%%%%%
%
%	Homogeneous Linear IVP Estimates
%
%%%%%%%%%%%%%%%%%%%%%%%%
%
%
%
%
\subsection{Sobolev-type estimates for the homogeneous linear IVP} 
We begin the analysis of the  components of \eqref{nh-flh-ibvp-sup} by estimating the solution of the homogeneous linear IVP \eqref{nh-lh-ivp} in Sobolev spaces.
\begin{theorem}[\b{Sobolev-type estimates for the homogeneous linear IVP}]
\label{nh-ivp-t}
The solution $U=S\big[U_0; 0\big]$ of the linear heat IVP \eqref{nh-lh-ivp} given by formula \eqref{nh-lh-ivp-sol} admits the estimates
\ddd{
\no{U}_{C([0, T]; H_x^s(\mathbb R))}
&\leqslant
\no{U_0}_{H_x^s(\mathbb R)},  && s\in \mathbb R,
\label{nh-ivp-se}
\\
\no{U}_{C(\mathbb R_x; H_t^{\frac{2s+1}{4}}(0,T))}
&\leqslant
c_s \no{U_0}_{H_x^s(\mathbb R)}, \quad &-&\tfrac 12 \leqslant s<\tfrac 32.
\label{nh-ivp-te}
}
\end{theorem}

\begin{remark}
The time estimate \eqref{nh-ivp-te} essentially dictates the space for the boundary datum $g_0$ of the half-line problem \eqref{nh-rd-ibvp}.
It is interesting to note that, despite the diffusive nature of the heat equation, the Sobolev exponent $\frac{2s+1}{4}$ is precisely the one appearing in the corresponding estimate for the (dispersive) linear Schr\"odinger equation (see \cite{kpv1991} and \cite{fhm2017}).
\end{remark}

\begin{proof}[\textnormal{\textbf{Proof of Theorem \ref{nh-ivp-t}.}}]
The solution formula \eqref{nh-lh-ivp-sol} combined with the definition of the $H_x^s(\mathbb R)$-norm imply the space estimate \eqref{nh-ivp-se} for all $s\in\mathbb R$ and all $t\geqslant 0$:
\ddd{
\no{U(t)}_{H_x^s(\mathbb R)}^2
&=
\int_{\xi\in\mathbb R} \left(1+\xi^2\right)^s \big| e^{-\xi^2 t} \widehat U_0(\xi)\big|^2 d\xi
\leqslant
\int_{\xi\in\mathbb R} \left(1+\xi^2\right)^s \big|\widehat U_0(\xi)\big|^2 d\xi
=
\no{U_0}_{H_x^s(\mathbb R)}^2. 
\nn
}
For continuity in time, we note that for any sequence 
$\left\{ t_n \right\} \subset [0, T]$ converging to $t\in [0, T]$ we have
\eee{\nn
\no{U(t_n) - U(t)}_{H_x^s(\mathbb R)}^2
=
\int_{\xi\in\mathbb R} \left(1+\xi^2\right)^s \left| 
\big(e^{-\xi^2 t_n} - e^{-\xi^2 t}\big)   \what U_0(\xi) \right|^2 d\xi
\lesssim
\int_{\xi\in\mathbb R} \left(1+\xi^2\right)^s \big|\what U_0(\xi)\big|^2 d\xi
<\infty.
}
Thus, by the dominated convergence theorem we infer  that $\lim_{n\to\infty}\no{U(t_n) - U(t)}_{H_x^s(\mathbb R)}^2 = 0$.

The proof of the time estimate \eqref{nh-ivp-te} is more involved.
Letting $m=\frac{2s+1}{4}$ and noting that  for  $-\tfrac 12 \leqslant s < \tfrac 32$ we have $0\leqslant m <1$, we employ the physical space definition of the Sobolev $H_t^m(0, T)$-norm:
\eee{\label{nh-hm-frac-norm-def}
\no{U(x)}_{H_t^m(0, T)}
=
\big\|U(x)\big\|_{L_t^2(0, T)} 
+
\big\|U(x)\big\|_m, \quad 0\leqslant m<1,
}
where for $0<m<1$ the fractional norm $\no{\cdot}_m$ is defined by
\eee{\nn
\no{U(x)}_{m}^2
=
\int_{t_1=0}^T \int_{t_2=0}^T
\frac{\left|U(x, t_1)-U(x, t_2)\right|^2}{\left|t_1-t_2\right|^{1+2m}} dt_2dt_1
\simeq
\int_{t=0}^T \int_{z=0}^{T-t} \frac{\left|U(x, t+z) - U(x, t)\right|^2}{z^{1+2m}} \, dz dt.
}
\noindent
\textit{\b{The norm $\no{U(x)}_{L_t^2(0, T)}$.}}  
We have
\sss{
\ddd{
\no{U(x)}_{L_t^2(0, T)}
%&\lesssim
%\left(\int_{t=0}^T \left|
%\frac{1}{2\pi}
%\int_{\xi\in\mathbb R} e^{i\xi x-\xi^2 t} \widehat U_0(\xi) d\xi
%\right|^2 dt \right)^{\frac 12}
%\nn\\
&\lesssim
\int_{t=0}^T
\left[
\int_{\xi=0}^1  e^{-\xi^2 t} 
\left(
\big|\widehat U_0(-\xi)\big|  
+
\big|\widehat U_0(\xi)\big| 
\right) d\xi
\right]^2
dt
\label{nh-te-split-1-T}
\\
&\quad
+
\int_{t=0}^T
\left[
\int_{\xi=1}^\infty  e^{-\xi^2 t} 
\left(
\big|\widehat U_0(-\xi)\big|  
+
\big|\widehat U_0(\xi)\big| 
\right) d\xi
\right]^2
dt.
\label{nh-te-split-2-T}
}
}
By the Cauchy-Schwarz inequality, we find
\eee{\label{nh-te-split-1-T-te}
\eqref{nh-te-split-1-T}
\lesssim
\int_{t=0}^T
\left(
\int_{\xi=0}^1  e^{-2\xi^2 t} 
\left(1+\xi^2\right)^{-s} d\xi
\right) \no{U_0}_{H_x^s(\mathbb R)}^2 dt
\leqslant
c_s T \no{U_0}_{H_x^s(\mathbb R)}^2, \quad s\in\mathbb R.
}

Next, making the change of variable $\xi = \sqrt \tau$ we have
\eee{
\eqref{nh-te-split-2-T}
=
\int_{t=0}^T \left(
\int_{\tau=1}^\infty  e^{-\tau t} \, \frac{\big|\widehat U_0(-\sqrt \tau)\big|+\big|\widehat U_0(\sqrt \tau)\big|}{2\sqrt \tau} \, d\tau
\right)^2 dt
\lesssim
\no{
\mathcal L\left\{\phi\right\}
}_{L_t^2(0, \infty)}^2,
\label{nh-te-tlap}
}
where 
$
\mathcal L\left\{\phi\right\}(t)
\doteq
\int_{\tau=0}^\infty e^{-\tau t} \phi(\tau) d\tau
$
is the  Laplace transform of the function  
\eee{\label{nh-phi-lap-def}
\phi(\tau) 
=
\bigg\{
\begin{array}{ll}
\tau^{-\frac 12} \big(\big|\widehat U_0(-\sqrt \tau)\big|+\big|\widehat U_0(\sqrt \tau)\big|\big), & \tau\geqslant 1,
\\
0, & 0\leqslant \tau <1.
\end{array}
}
\begin{lemma}[\b{$L^2$-boundedness of the Laplace transform}]
\label{nh-tlap-l}
Suppose $\phi\in L_\tau^2(0,\infty)$. Then, the map
$$
\mathcal L: \phi  \mapsto  \int_{\tau=0}^{\infty} e^{-\tau t} \phi(\tau) d\tau
$$
is bounded from $L_\tau^2(0,\infty)$ into $L_t^2(0, \infty)$ with
$$
\no{\mathcal L\left\{\phi\right\}}_{L_t^2(0,\infty)}
\leqslant
\sqrt{\pi} \no{\phi}_{L_\tau^2(0,\infty)}.
$$
\end{lemma}
A proof of Lemma \ref{nh-tlap-l} is available in \cite{fhm2017}. Upon employing this lemma for $\phi$ given by \eqref{nh-phi-lap-def}, \eqref{nh-te-tlap}  yields
\eee{
\eqref{nh-te-split-2-T}
\lesssim
\no{
\tau^{-\frac 12} \big(\big|\widehat U_0(-\sqrt \tau)\big|+\big|\widehat U_0(\sqrt \tau)\big|\big)
}_{L_\tau^2(1, \infty)}^2
\lesssim
\int_{\tau=1}^\infty
\tau^{-1} \big|\widehat U_0(-\sqrt \tau)\big|^2  d\tau
+
\int_{\tau=1}^\infty
\tau^{-1}  \big|\widehat U_0(\sqrt \tau)\big|^2 d\tau.
\nn
}
Thus, letting $\xi = -\sqrt \tau$  and $\xi = \sqrt \tau$ in the first and the second integral respectively, we obtain
\eee{
\eqref{nh-te-split-2-T}
\lesssim
\int_{|\xi|\geqslant 1}
|\xi|^{-1} \big|\widehat U_0(\xi)\big|^2   d\xi
\leqslant
\no{U_0}_{H_x^{-\frac 12}(\mathbb R)}^2
\leqslant
\no{U_0}_{H_x^{s}(\mathbb R)}^2, \quad s\geqslant -\tfrac 12.
\label{nh-te-split-2-T-te}
}
Combining estimates \eqref{nh-te-split-1-T-te} and \eqref{nh-te-split-2-T-te},  we find 
\eee{\label{L2-te-T}
\no{U(x)}_{L_t^2(0, T)}
\lesssim
\no{U_0}_{H_x^s(\mathbb R)}, \quad s\geqslant -\tfrac 12, \ x\in\mathbb R.
}
\noindent
\textit{\b{The fractional norm $\no{U(x)}_m$.}} 
Recall that now $-\frac 12<s<\frac 32$.
Starting from  \eqref{nh-lh-ivp-sol}, we compute
\ddd{
&\quad
\left|U(x, t+z) - U(x, t)\right|
\lesssim
\int_{\xi\in\mathbb R}  e^{-\xi^2 t} \big(1-e^{-\xi^2 z}\big)  \big|\widehat U_0(\xi)\big| d\xi
\nn\\
&=
\int_{\xi=0}^1 e^{-\xi^2 t} \big(1-e^{-\xi^2 z}\big) \left(\big|\widehat U_0(-\xi)\big| + \big|\widehat U_0(\xi)\big| \right) d\xi
+
\int_{\xi=1}^\infty e^{-\xi^2 t} \big(1-e^{-\xi^2 z}\big) \left(\big|\widehat U_0(-\xi)\big| + \big|\widehat U_0(\xi)\big| \right) d\xi.
\nn
}
Hence, from the definition of the fractional Sobolev norm we have
\sss{
\ddd{
\no{U(x)}_{m}^2
&\lesssim
\int_{t=0}^T\int_{z=0}^{T-t}  \frac{1}{z^{1+2m}}  \left[\int_{\xi=0}^1 e^{-\xi^2 t} \big(1-e^{-\xi^2 z}\big) \left(\big|\widehat U_0(-\xi)\big| + \big|\widehat U_0(\xi)\big| \right) d\xi\right]^2 dz dt
\label{nh-te-frac-split-1}
\\
&\quad
+
\int_{t=0}^T\int_{z=0}^{T-t}  \frac{1}{z^{1+2m}}   \left[\int_{\xi=1}^\infty e^{-\xi^2 t} \big(1-e^{-\xi^2 z}\big) \left(\big|\widehat U_0(-\xi)\big| + \big|\widehat U_0(\xi)\big| \right) d\xi\right]^2 dz dt.
\label{nh-te-frac-split-2}
}
}
For any $s\in\mathbb R$, the Cauchy-Schwarz inequality implies
\eee{
\eqref{nh-te-frac-split-1}
\lesssim
T \no{U_0}_{H_x^s(\mathbb R)}^2
\int_{\xi=0}^1  \left(1+\xi^2\right)^{-s}
\int_{z=0}^{T}   \frac{\big(1-e^{-\xi^2 z}\big)^2}{z^{1+2m}} \,  dz  d\xi. 
\nn
}
Moreover, 
\eee{\label{nh-gamma-est}
\int_{z=0}^T \frac{\big(1-e^{-\xi^2 z}\big)^2}{z^{1+2m}}\, dz
\leqslant
\xi^{4m}
 \int_{\zeta=0}^\infty \frac{\left(1-e^{-\zeta}\right)^2}{\zeta^{1+2m}}\,  d\zeta
 \simeq
 \xi^{4m},
}
since the $\zeta$-integral above converges for $0<m<1$. Hence,  
\ddd{\label{nh-te-frac-split-1-te}
\eqref{nh-te-frac-split-1}
\lesssim
T \no{U_0}_{H_x^s(\mathbb R)}^2
\int_{\xi=0}^1  \left(1+\xi^2\right)^{-s} \xi^{4m}  d\xi
=
c_s T \no{U_0}_{H_x^s(\mathbb R)}^2, \quad -\tfrac 12 < s < \tfrac 32.
}

Furthermore, letting $\xi = \sqrt \tau$  we have
\eee{
\eqref{nh-te-frac-split-2}
\lesssim
\int_{z=0}^T \frac{1}{z^{1+2m}}   
\no{
\mathcal L\left\{\phi\right\}(z)}_{L_t^2(0, \infty)}^2 dz,
\nn
}
where
$$
\phi(\tau) 
=
\bigg\{
\begin{array}{ll}
\tau^{-\frac 12} \left(1-e^{-\tau z}\right) \big(\big|\widehat U_0(-\sqrt \tau)\big| + \big|\widehat U_0(\sqrt \tau)\big| \big), & \tau\geqslant 1,
\\
0, & 0\leqslant \tau <1.
\end{array}
$$
Thus, by the Laplace transform bound of Lemma \ref{nh-tlap-l}  we find
\eee{
\eqref{nh-te-frac-split-2}
\lesssim
\int_{\xi=1}^\infty 
\xi^{-1}  \left(\big|\widehat U_0(-\xi)\big| + \big|\widehat U_0(\xi)\big| \right)^2 
\int_{z=0}^T  \frac{\big(1-e^{-\xi^2 z}\big)^2}{z^{1+2m}}  \, dz  d\xi.
\nn
 }
Hence, using estimate \eqref{nh-gamma-est} for the $z$-integral, we obtain
\eee{
\eqref{nh-te-frac-split-2}
\lesssim
\int_{\xi=1}^\infty 
\xi^{4m-1}  \left(\big|\widehat U_0(-\xi)\big| + \big|\widehat U_0(\xi)\big| \right)^2  d\xi
\lesssim
\no{U_0}_{H_x^s(\mathbb R)}^2, \quad -\tfrac 12 < s < \tfrac 32.
\label{nh-te-frac-split-2-te}
}
Note that since $\xi\geqslant 1$ we have $\xi^{2s}\lesssim \left(1+\xi^2\right)^s$ even for negative values of $s$. 

Combining estimates \eqref{nh-te-frac-split-1-te} and \eqref{nh-te-frac-split-2-te}, we deduce 
\eee{\label{frac-te-T}
\no{U(x)}_m \lesssim \no{U_0}_{H_x^s(\mathbb R)}, \quad  -\tfrac 12 < s < \tfrac 32, \ x\in \mathbb R.
}
In turn, estimates \eqref{L2-te-T} and \eqref{frac-te-T} combined with the definition \eqref{nh-hm-frac-norm-def}   yield  estimate \eqref{nh-ivp-te}.

Continuity with respect to $x$ follows from estimating the $H_t^{\frac{2s+1}{4}}(0, T)$-norm of the difference  $U(x_n)-U(x)$  as above  for any sequence $\{x_n\}\to x$ and then using the dominated convergence theorem to show that this norm vanishes in the limit $n\to\infty$.
\end{proof}
%
%
%
%%%%%%%%%%%%%%%%%%%%%%%%  
%
%	       Forced Linear IVP Estimates
%
%%%%%%%%%%%%%%%%%%%%%%%%  
%
%
%
\subsection{Sobolev-type estimates for the forced linear IVP}\label{nh-fivp-s}
Next, we shall obtain estimates in Sobolev spaces for the second component of the superposition \eqref{nh-flh-ibvp-sup}, namely the forced linear IVP \eqref{nh-flh-ivp}. 
As we will see below, the lack of algebra property in $H_x^s(\mathbb R)$ when $s<\frac 12$ forces us to estimate the solution of problem \eqref{nh-flh-ivp} in a different way than for $s>\frac 12$, giving rise to the space 
\eee{\label{nh-cbeta-def-line}
C^\alpha([0, T]; L_x^p(\mathbb R))
=
\Big\{
u \in C([0, T]; L_x^p(\mathbb R)): 
\sup_{t\in [0,T]} t^\alpha \no{u(t)}_{L_x^p(\mathbb R)} < \infty
\Big\}
}
and thereby introducing the space $C^\alpha([0, T]; L_x^p(0, \infty))$ in the ``rough'' data solution space $Y$ of Theorem \ref{nh-ibvp-wp-l-t}.
We note that spaces of the type \eqref{nh-cbeta-def-line} also appear in the analysis of the reaction-diffusion IVP  \cite{w1980}.

\begin{theorem}[\b{Sobolev-type estimates for the forced linear IVP}]
\label{nh-fivp-t}
For $F=\prod_{j=1}^p F_j$, the solution $W=S\big[0; F\big]$ of the forced linear heat IVP \eqref{nh-flh-ivp} given by formula \eqref{nh-flh-ivp-sol-comb} admits the space estimates
\sss{\label{nh-fivp-se}
\ddd{
\no{W}_{C([0, T]; H_x^s(\mathbb R))}
&\lesssim
T \prod_{j=1}^p  \no{F_j}_{C([0, T]; H_x^s(\mathbb R))},
&& s>\tfrac 12,
\label{nh-fivp-se-h}
\\
\no{W}_{C([0, T]; H_x^s(\mathbb R))}
&\lesssim
\sqrt T \prod_{j=1}^p
\no{F_j}_{C^\alpha([0, T]; L_x^p(\mathbb R))},
\quad &&
0\leqslant s < \tfrac 12, 
\label{nh-fivp-se-l}
}
}
and the time estimates
\sss{\label{nh-fivp-te}
\ddd{
\no{W}_{C(\mathbb R_x; H_t^{\frac{2s+1}{4}}(0, T))}
&\lesssim
\sqrt T  
\prod_{j=1}^p
\no{F_j}_{C([0, T]; H_x^s(\mathbb R))},
&&\tfrac 12 <s < \tfrac 32,
\label{nh-fivp-te-h}
\\
\no{W}_{C(\mathbb R_x; H_t^{\frac{2s+1}{4}}(0, T))}
&\lesssim
\sqrt T\,
\prod_{j=1}^p
\no{F_j}_{C^{\alpha}([0, T]; L_x^p(\mathbb R))},
\quad 
&-&\tfrac 12 \leqslant s < \tfrac 12.
\label{nh-fivp-te-l}
}
}
\end{theorem}

\begin{proof}[\textnormal{\textbf{Proof of Theorem \ref{nh-fivp-t}.}}]
\textit{\b{Space estimate \eqref{nh-fivp-se-h}.}}
Using Minkowski's integral inequality and estimate \eqref{nh-ivp-se}, we have
\eee{\label{nh-mink-temp}
\no{W(t)}_{H_x^s(\mathbb R)}
\leqslant
\int_{t'=0}^t \no{S\big[F(t'); 0\big](t-t')}_{H_x^s(\mathbb R)} dt'
\leqslant
T \no{F}_{C([0, T]; H_x^s(\mathbb R))}, \quad s\in\mathbb R,
}
which implies the space estimate \eqref{nh-fivp-se-h} via the algebra property in $H_x^s(\mathbb R)$ for $s>\tfrac 12$.

Regarding continuity in time, for any sequence 
$\left\{ t_n \right\} \subset [0, T]$ converging to $t\in [0, T]$ we have
{\small
\ddd{
\no{W(t_n) - W(t)}_{H_x^s(\mathbb R)}
\leqslant
\!
\int_{t'=0}^T\!
\left(
\int_{\xi\in\mathbb R}\!\! \left(1+\xi^2\right)^s \left| \chi_{(0, t_n)}(t') \, e^{-\xi^2 (t_n-t')} - \chi_{(0, t)}(t') \, e^{-\xi^2 (t-t')}\right|^2 \left| \what F(\xi, t') \right|^2 d\xi
\right)^{\!\frac 12}
\!\! dt'.
\nn
}
}
The $\xi$-integral in the above inequality can be bounded above by $\no{F(t')}_{H_x^s(\mathbb R)}$ so, in turn, we find 
$\no{W(t_n) - W(t)}_{H_x^s(\mathbb R)}
\lesssim
T \sup_{t\in [0, T]} \no{F(t)}_{H_x^s(\mathbb R)}<\infty.
$
Hence, using the dominated convergence theorem for the $t'$- and $\xi$-integrals, we conclude that 
$
\lim_{n\to\infty}\no{W(t_n) - W(t)}_{H_x^s(\mathbb R)} =0.
$
\vskip 3mm
\noindent
\textit{\b{Space estimate \eqref{nh-fivp-se-l}.}}
The proof of this estimate is more complicated because it concerns smaller values of $s$ for which the algebra property is not available. Instead, starting from the first inequality in \eqref{nh-mink-temp}, applying Minkowski's integral inequality and taking $\sup$ in $\xi$, we have
\eee{\label{temp8-nh}
\no{W(t)}_{H_x^s(\mathbb R)}
\lesssim
\int_{t'=0}^t
\left(\int_{\xi=0}^\infty \left(1+\xi^2\right)^{s}  e^{-2(t-t')\xi^2} d\xi\right)^{\frac{1}{2}}
\no{F(t')}_{L_x^1(\mathbb R)}  dt'.
}
Letting $z=2(t-t')\xi^2$ and noting that 
$(a+b)^s \leqslant 2^s \left(a^s + b^s \right)$ for 
$a, b, s \geqslant 0$, we find
\ddd{
\left(\int_{\xi\in\mathbb R} \left(1+\xi^2\right)^{s}  e^{-2(t-t')\xi^2} d\xi\right)^{\frac{1}{2}}
&\lesssim
\bigg[
\left(t-t'\right)^{-\frac 12}\int_{z=0}^\infty z^{-\frac 12} e^{-z} dz
+
\left(t-t'\right)^{-s-\frac 12}\int_{z=0}^\infty  e^{-z}  z^{s-\frac 12}  dz \bigg]^{\frac 12}
\nn\\
&\lesssim
\left(t-t'\right)^{-\frac{1}{4}} + \left(t-t'\right)^{-\frac{1}{4}-\frac{s}{2}}
\lesssim
\left(t-t'\right)^{-\frac{1}{4}-\frac{s}{2}}, \quad s\geqslant 0,
\nn
}
where we have also used the fact that $0\leqslant t-t' \leqslant t\leqslant T < 1$ and 
$0\leqslant \frac{1}{4} \leqslant \frac{1}{4}+\frac{s}{2}$ for  
$s\geqslant 0$.
Returning to \eqref{temp8-nh}, we have
\eee{
\!\!
\no{W(t)}_{H_x^s(\mathbb R)}
\lesssim
\!
\int_{t'=0}^t\!\!
\left(t-t'\right)^{-\frac{1}{4}-\frac{s}{2}} 
\!
\no{F(t')}_{L_x^1(\mathbb R)} dt'
\leqslant
\!
\sup_{t\in[0, T]}
\!
\left[
t^{p\gamma}  \no{F(t)}_{L_x^1(\mathbb R)}
\right]
\!
\int_{t'=0}^t \!\! \left(t-t'\right)^{-\frac{1}{4}-\frac{s}{2}} 
\! (t')^{-p\gamma} dt'
\nn
}
for some $\gamma$ to be specified.
Hence, using the estimate
\ddd{\label{beta-function-l-nh}
\int_{t'=0}^t \left(t-t'\right)^{-\beta_1} (t')^{-\beta_2} dt'
&=
t^{1-\beta_1-\beta_2}
\int_{\eta=0}^1 \left(1-\eta\right)^{-\beta_1} \eta^{-\beta_2} d\eta
\simeq 
t^{1-\beta_1-\beta_2}, \quad \beta_1, \beta_2 <1,
}
for $\beta_1= \frac{1}{4} + \frac{s}{2}$, $\beta_2= p\gamma$ and then applying  H\"older's inequality after setting $F =  \prod_{j=1}^p F_j$,  we get
\eee{\nn
\no{W(t)}_{H_x^s(\mathbb R)}
\lesssim
t^{\frac 12\left(\frac 32-s\right)-p\gamma}
\prod_{j=1}^p  \no{F_j}_{C^\gamma([0, T]; L_x^{p}(\mathbb R))},
\quad
s<\tfrac{3}{2}, \ \gamma<\tfrac 1p.
}
Finally, choosing  $\gamma = \alpha = \frac{1}{p}\left(\frac 12- b\right)$ with $\frac{2s+1}{4} < b < \frac 12$   we have
$
\tfrac 12\left(\tfrac 32-s\right)-p\gamma
=
b-\tfrac{2s-1}{4}
>
\tfrac 12.
$
Hence, recalling that $0\leqslant t <1$, we deduce the space estimate \eqref{nh-fivp-se-l}. Observe that the condition $b < \frac 12$ (which ensures that $\alpha >0$) combined with the condition $\frac{2s+1}{4}<b$ further restricts $s<\frac 12$.
\vskip 3mm
\noindent
\textit{\b{Time estimate \eqref{nh-fivp-te-h}.}}
For $m=\frac{2s+1}{4}$ with $m\in [0, 1) \Leftrightarrow s\in\left[-\frac 12, \frac 32\right)$, we employ the physical space definition \eqref{nh-hm-frac-norm-def} of the  $H_t^m(0, T)$-norm.
Thus, we need to estimate $\no{W(x)}_{L_t^2(0, T)}$ and $\no{W(x)}_{m}$.

Starting from the Duhamel representation \eqref{nh-flh-ivp-sol-d} and combining Minkowski's integral inequality with the time estimate \eqref{nh-ivp-te}, we find
\eee{\label{flh-ivp-te-l2}
\no{W(x)}_{L_t^2(0, T)}
\leqslant
\int_{t'=0}^T \no{S\big[F(\cdot, t'); 0 \big](x, t-t')}_{L_t^2(0, T)} dt'
\leqslant
T \sup_{t\in \left[0, T\right]} \no{F(t)}_{H_x^s(\mathbb R)}.
}

Moreover, writing 
\ddd{
W(x, t+z) - W(x, t)
&=
\int_{t'=0}^t 
\Big[
S\big[F(\cdot, t'); 0 \big] (x, t+z-t') 
-
S\big[F(\cdot, t'); 0 \big] (x, t-t')
\Big] dt'
\nn\\
&\quad
+
\int_{t'=t}^{t+z} S\big[F(\cdot, t'); 0 \big] (x, t+z-t') dt'
\nn
}
we have
\sss{\label{fls-ivp-te-b}
\ddd{
\hskip -.5cm
\no{W(x)}_m^2
&\lesssim
\int_{t=0}^T \int_{z=0}^{T-t} \frac{1}{z^{1+2m}} 
\left| \displaystyle \int_{t'=0}^t 
\Big[
S\big[F(\cdot, t'); 0 \big] (x, t+z-t') 
-
S\big[F(\cdot, t'); 0 \big] (x, t-t')
\Big] dt' \right|^2 dz dt
\label{fls-ivp-te-b1}
\\
&\quad
+
\int_{t=0}^T \int_{z=0}^{T-t}  \frac{1}{z^{1+2m}} 
\left| \int_{t'=t}^{t+z} S\big[F(\cdot, t'); 0 \big] (x, t+z-t') dt'\right|^2  dz dt.
\label{fls-ivp-te-b2}
}
}
By Minkowski's integral inequality and estimate \eqref{nh-ivp-te}, we find
\eee{
\eqref{fls-ivp-te-b1}
\leqslant
\left(\int_{t'=0}^T 
\Big\|
S\big[F(\cdot, t'); 0 \big] (x, t-t')\Big\|_m dt'
\right)^2
\lesssim
\left(T \sup_{t\in \left[0, T\right]} \no{F(t)}_{H_x^s(\mathbb R)}
\right)^2.
\label{fls-ivp-te-b1-est}
}
Furthermore, recalling the representation \eqref{nh-flh-ivp-sol}  we have
\eee{
\eqref{fls-ivp-te-b2}
\leqslant
\int_{t=0}^T \int_{z=0}^{T-t}
\frac{1}{z^{1+2m}}
\no{
\frac{1}{2\pi} \int_{\xi\in \mathbb R} e^{i\xi x}
\int_{t'=t}^{t+z} e^{-\xi^2(t+z-t')}  \widehat F(\xi, t')  dt' d\xi}_{L_x^\infty(\mathbb R)}^2 dz dt.
\nn
}
Hence, the Sobolev embedding theorem for $s>\frac 12$ implies
\ddd{
\eqref{fls-ivp-te-b2}
&\leqslant
\int_{t=0}^T \int_{z=0}^{T-t} \frac{1}{z^{1+2m}}
\no{
\frac{1}{2\pi} \int_{\xi\in \mathbb R} e^{i\xi x} 
\int_{t'=t}^{t+z} e^{-\xi^2(t+z-t')}  \widehat F(\xi, t')  dt' d\xi}_{H_x^s(\mathbb R)}^2  dz dt
\nn\\
&\leqslant
\int_{t=0}^T \int_{z=0}^{T-t} \frac{1}{z^{1+2m}}
\int_{\xi\in \mathbb R} \left(1+\xi^2\right)^s
\left(
\int_{t'=t}^{t+z} 
\left|\widehat F(\xi, t')\right|  dt'
\right)^2 
d\xi dz dt,
\nn
}
and Minkowski's integral inequality further yields
\ddd{
\eqref{fls-ivp-te-b2}
\leqslant
\int_{t=0}^T \int_{z=0}^{T-t}
\sup_{t'\in \left[t, t+z\right]} \no{F(t')}_{H_x^s(\mathbb R)}^2 \cdot z^{1-2m}  dz dt
\lesssim
T^{3-2m} \sup_{t\in \left[0, T\right]} \no{F(t)}_{H_x^s(\mathbb R)}^2,
\label{fls-ivp-te-b2-est}
}
after recalling that $m<1$.
Combining  \eqref{fls-ivp-te-b}, \eqref{fls-ivp-te-b1-est}, \eqref{fls-ivp-te-b2-est} and the fact that $T<1$, we deduce
\eee{
\no{W(x)}_m
\lesssim
\sqrt T  \sup_{t\in \left[0, T\right]} \no{F(t)}_{H_x^s(\mathbb R)}.
\label{fls-ivp-te-b-est}
}
Overall, estimates \eqref{flh-ivp-te-l2}  and \eqref{fls-ivp-te-b-est}  together with the definition \eqref{nh-hm-frac-norm-def} imply 
$$
\no{W(x)}_{H_t^m(0, T)}
\lesssim
\sqrt T \sup_{t\in \left[0, T\right]} \no{F(t)}_{H_x^s(\mathbb R)},\quad \tfrac 12<s < \tfrac 32,
$$
from which we deduce the time estimate \eqref{nh-fivp-te-h} via the algebra property in $H_x^s(\mathbb R)$ since $s>\frac 12$.

\vskip 3mm
\noindent
\textit{\b{Time estimate \eqref{nh-fivp-te-l}.}}
As for time estimate  \eqref{nh-fivp-te-h}, we restrict $-\frac 12 \leqslant s < \frac 32$ and use the norm \eqref{nh-hm-frac-norm-def}, which involves $\no{W(x)}_{L_t^2(0, T)}$ and $\no{W(x)}_{m}$. 

Starting from formula  \eqref{nh-flh-ivp-sol} and taking $\sup$ in $\xi$, we have
\eee{
\no{W(x)}_{L_t^2(0, T)}^2
\lesssim
\int_{t=0}^T 
\left[
\int_{t'=0}^t
(t-t')^{-\frac{1}{2}}
\no{F(t')}_{L_x^1(\mathbb R)}
dt'
\right]^2 dt.
\nn
}
Thus, using  \eqref{beta-function-l-nh} with 
$\beta_1 = \beta_2 = \frac 12$ and  H\"older's inequality in $x$, we obtain
\eee{\label{fivp-te-l2-nh}
\no{W(x)}_{L_t^2(0, T)}
\lesssim
\sqrt T   \sup_{t\in[0, T]}  t^{\frac 12}
\no{F(t)}_{L_x^1(\mathbb R)}
\lesssim
\sqrt T 
\prod_{j=1}^p
\sup_{t\in[0, T]}  
\left(t^{\frac{1}{2p}}
\no{F_j(t)}_{L_x^p(\mathbb R)}\right).
}

For the fractional norm $\no{W(x)}_m$, we use formula \eqref{nh-flh-ivp-sol} to write
\sss{
\ddd{
\no{W(x)}_m^2
&\lesssim
\int_{t=0}^T \int_{z=0}^{T-t} \frac{1}{z^{1+2m}} 
\left| 
\int_{t'=0}^t \int_{\xi\in\mathbb R} e^{i\xi x-\xi^2(t-t')} \big(e^{-\xi^2 z}-1\big)  \widehat{F}(\xi, t') d\xi dt' 
\right|^2 dz dt
\label{fh-ivp-te-p1}
\\
&\quad
+\int_{t=0}^T \int_{z=0}^{T-t} \frac{1}{z^{1+2m}} 
\left|
\int_{t'=t}^{t+z}  \int_{\xi\in\mathbb R} e^{i\xi x-\xi^2(t+z-t')}  \widehat{F}(\xi, t') d\xi dt' \right|^2 dz dt.
\label{fh-ivp-te-p2}
}
}
Minkowski's integral inequality together with the bound \eqref{nh-gamma-est} imply
\eee{
\eqref{fh-ivp-te-p1}
\lesssim
\int_{t=0}^T 
\left[
 \int_{t'=0}^t 
\left(\int_{\xi=0}^\infty 
e^{-\xi^2(t-t')} \xi^{2m} d\xi \right)
\no{F(t')}_{L_x^1(\mathbb R)}
dt'\right]^2 dt.
\nn
}
Estimating the $\xi$-integral as in the proof of estimate \eqref{nh-fivp-se-l} and recalling that $m=\frac{2s+1}{4}$, we find
\eee{
\eqref{fh-ivp-te-p1}
\lesssim
\int_{t=0}^T 
\left[
\int_{t'=0}^t 
\left(t-t'\right)^{- \frac s2 -\frac 34}
\no{F(t')}_{L_x^1(\mathbb R)}
dt'\right]^2 dt.
\nn
}
Note  that the singularity at $t'=t$ is integrable provided that $s<\tfrac 12$. 
Then, using the bound \eqref{beta-function-l-nh} with $\beta_1 = \frac s2 + \frac 34$ and $\beta_2 = 1-\beta_1$, we obtain
\eee{
\eqref{fh-ivp-te-p1}
\lesssim
\int_{t=0}^T 
 \left(\sup_{t'\in[0, t]} (t')^{1-\left(\frac s2 + \frac 34\right)} 
\no{F(t')}_{L_x^1(\mathbb R)}\right)^2   dt
\leqslant
T   \left(\sup_{t\in[0, T]} t^{\frac 12 \left(\frac 12  - s\right)} 
\no{F(t)}_{L_x^1(\mathbb R)}\right)^2.
\label{fh-ivp-te-p1-est}
}
For the term \eqref{fh-ivp-te-p2}, taking $\sup$ in $\xi$ and noting that $\int_{\xi\in\mathbb R} e^{-\xi^2(t+z-t')}  d\xi 
\simeq
\left(t+z-t'\right)^{-\frac{1}{2}}$ we have
\ddd{
&\eqref{fh-ivp-te-p2}
\lesssim
\int_{t=0}^T \int_{z=0}^{T-t} \frac{1}{z^{1+2m}} 
\left(
\int_{t'=t}^{t+z} 
\left(t+z-t'\right)^{-\frac{1}{2}}
\no{F(t')}_{L_x^1(\mathbb R)}
dt' \right)^2 dz dt
\nn\\
&\!\leqslant
\!
\int_{t=0}^T  \int_{z=0}^{T-t} \!\! \frac{1}{z^{1+2m}} 
\bigg(
\int_{t'=t}^{t+z}    \!\!
\left(t+z-t'\right)^{-\frac{1}{2}}
\left(t'-t\right)^{-\left(\frac{1}{2}-b\right)}
dt'
\bigg)^{\!2} 
\bigg(\sup_{t'\in [t, t+z]} \!\! \left(t'-t\right)^{\frac{1}{2}-b}
\no{F(t')}_{L_x^1(\mathbb R)}
\!\bigg)^{\!2} \! dz dt
 \nn
}
for some $b$ to be determined. Then,  \eqref{beta-function-l-nh} with  $\beta_1 = \frac{1}{2}$, $\beta_2 = \frac{1}{2} - b$ and $b>-\frac{1}{2}$ implies
\eee{
\eqref{fh-ivp-te-p2}
\lesssim
T^{2(b-m)+1}
\left(\sup_{t\in [0, T]} t^{\frac{1}{2}-b}
\no{F(t)}_{L_x^1(\mathbb R)}
 \right)^2, \quad b>m, \ b<  \tfrac{1}{2}.
 \label{fh-ivp-te-p2-est}
}
Altogether, and since $m=\frac{2s+1}{4}\geqslant 0$, estimate \eqref{fh-ivp-te-p2-est} is valid for  $\tfrac{2s+1}{4} < b < \frac 12$ and $-\tfrac 32 < s < \tfrac 12$.
Furthermore, since $t\leqslant T<1$, using H\"older's inequality estimates   \eqref{fh-ivp-te-p1-est}  and \eqref{fh-ivp-te-p2-est} combine to 
\eee{
\no{W(x)}_m
\lesssim
\sqrt T
\prod_{j=1}^p
\left(
\sup_{t\in[0, T]} 
t^{\frac 1p\left(\frac 12-b\right)}
\,
 \no{F_j(t)}_{L_x^p(\mathbb R)}
 \right),
}
which, together estimate  \eqref{fivp-te-l2-nh} and the fact that $b\geqslant 0$ for $m\geqslant 0$, implies estimate \eqref{nh-fivp-te-l}.

Continuity in $x$ follows by estimating in the same way as above the $H_t^{\frac{2s+1}{4}}(0, T)$-norm of the difference $W(x_n)-W(x)$ for any sequence $\{x_n\} \to x$ and then employing the dominated convergence theorem to deduce that this norm vanishes in the limit $n\to\infty$.
\end{proof}

\subsection{Additional estimates} 
The Sobolev-type estimates of Theorem \ref{nh-fivp-t} introduce the need for estimating the solutions of the linear IVPs \eqref{nh-lh-ivp} and \eqref{nh-flh-ivp} in the space $C^\alpha([0, T]; L_x^{p}(\mathbb R))$ defined by \eqref{nh-cbeta-def-line}.
\begin{theorem}[\b{$C^\alpha([0, T]; L_x^{p}(\mathbb R))$-estimate for the homogeneous linear IVP}]\label{nh-ivp-ca-t}
Suppose $\frac 12 - \frac{1}{p} \leqslant s < \frac 12$. 
Then, the solution $U = S\big[U_0; 0\big]$ of the linear heat IVP \eqref{nh-lh-ivp} given by formula \eqref{nh-lh-ivp-sol} admits the estimate
\eee{\label{nh-ivp-ca-est}
\no{U}_{C^\alpha([0, T]; L_x^p(\mathbb R))}
\lesssim 
T^\alpha  \no{U_0}_{H_x^s(\mathbb R)}.
}
\end{theorem}

\begin{proof}[\textnormal{\textbf{Proof of Theorem \ref{nh-ivp-ca-t}}}]
The Sobolev inequality
$\no{f}_{L^p(\mathbb R)} \leqslant c \no{f}_{H^{\frac 12 - \frac 1p}(\mathbb R)}$, 
$2 \leqslant p < \infty$,
(cf.  \cite{lp-book}) implies 
$$
\no{U}_{C^\alpha([0, T]; L_x^{p}(\mathbb R))}
=
\sup_{t\in[0, T]} \left[t^\alpha \no{U(t)}_{L_x^{p}(\mathbb R)} \right]
\lesssim 
\sup_{t\in[0, T]} \Big[t^\alpha \no{U(t)}_{H_x^{\frac 12 - \frac{1}{p}}(\mathbb R)}\Big].
$$
Moreover, for $s \geqslant \tfrac 12 - \tfrac{1}{p}$ the space estimate \eqref{nh-ivp-se} yields
$
\no{U(t)}_{H_x^{\frac 12 - \frac{1}{p}}(\mathbb R)}
\leqslant
\no{U_0}_{H_x^s(\mathbb R)}
$,
which gives in turn  estimate \eqref{nh-ivp-ca-est}.  Note that the restriction $s<\frac 12$ comes from the definition  \eqref{nh-cbeta-def} of $\alpha$ and is not required in the current proof.
Continuity in time follows  via the dominated convergence theorem as in Theorem \ref{nh-ivp-t}.
\end{proof}

We complete this section with the estimation of the forced linear IVP \eqref{nh-flh-ivp}.
\begin{theorem}[\b{$C^\alpha([0, T]; L_x^{p}(\mathbb R))$-estimate for the forced linear IVP}]\label{nh-fivp-ca-t}
For $F=\prod_{j=1}^p F_j$, the solution $W = S\big[0; F\big]$ of the forced linear heat IVP \eqref{nh-flh-ivp} given by formula \eqref{nh-flh-ivp-sol-comb}  satisfies the estimate
\eee{\label{nh-fivp-cat-est}
\no{W}_{C^{\alpha}([0, T]; L_x^{p}(\mathbb R))}
\lesssim 
T^{\alpha}
\prod_{j=1}^p \no{F_j}_{C^{\alpha}([0, T]; L_x^{p}(\mathbb R))}.
}
\end{theorem}

\begin{proof}[\textnormal{\textbf{Proof of Theorem \ref{nh-fivp-ca-t}.}}]
The Sobolev inequality $\no{f}_{L^p(\mathbb R)} \leqslant c \no{f}_{H^{\frac 12 - \frac 1p}(\mathbb R)}$, $2 \leqslant p < \infty$, implies
\eee{\nn
\no{W}_{C^\alpha([0, T]; L_x^{p}(\mathbb R))}
\lesssim 
\sup_{t\in[0, T]} \Big[t^\alpha \no{W(t)}_{H_x^{\frac 12 - \frac{1}{p}}(\mathbb R)}\Big],
}
from which we can obtain the desired estimate using the space estimate \eqref{nh-fivp-se-l} for $s=\frac 12 - \frac 1p$ (this restricts $p\geqslant 2$) and the fact that $\alpha>0$ and $T<1$. Continuity in time follows via the dominated convergence theorem as in Theorem \ref{nh-fivp-t}.
\end{proof}
%
%
%
%
%%%%%%%%%%%%%%%%%%%%%%%%  
%
%		Linear IBVP Estimates
%
%%%%%%%%%%%%%%%%%%%%%%%%  
%
%
%
\section{Linear IBVP Estimates on the Half-Line}
\label{nh-ibvp-le-s}

Having completed the estimation of IVPs \eqref{nh-lh-ivp}  and \eqref{nh-flh-ivp}, we turn our attention to the remaining two linear problems involved in the decomposition of Subsection \ref{nh-dec-ss}, namely the reduced IBVPs \eqref{nh-lh-ibvp-r} and \eqref{nh-flh-ibvp-r}. These two problems are of the same type, since 
(i) they are both homogeneous and with zero initial datum; (ii) their boundary data $G_0$ and $W|_{x=0}$ both vanish at $t=0$ for $s>\frac 12$; (iii) their boundary data $G_0$ and $W|_{x=0}$ belong to  $H_t^{\frac{2s+1}{4}}(0, T)$. Indeed, thanks to the time estimates \eqref{nh-ivp-te} and \eqref{nh-fivp-te-h} and the extension inequalities \eqref{nh-uU},   \eqref{nh-Ff-se} and \eqref{nh-Ff-ca-est}, we have
\sss{\label{nh-G0-W0-est}
\ddd{
\no{G_0}_{H_t^{\frac{2s+1}{4}}(0, T)}
&\lesssim
\no{u_0}_{H_x^s(0, \infty)}
+
\no{g_0}_{H_t^{\frac{2s+1}{4}}(0, T)}, \quad && 0 \leqslant s < \tfrac 32,
\label{nh-G0-est}
\\
\no{W|_{x=0}}_{H_t^{\frac{2s+1}{4}}(0, T)}
&\lesssim
\sqrt T  
 \prod_{j=1}^p \sup_{t\in [0, T]}   \no{f_j(t)}_{H_x^s(0, \infty)},
&&
\tfrac 12 <s \leqslant \tfrac 32,
\label{nh-W0-est-h}
\\
\no{W|_{x=0}}_{H_t^{\frac{2s+1}{4}}(0, T)}
&\lesssim
\sqrt T\,
\prod_{j=1}^p
\no{f_j}_{C^{\alpha}([0, T]; L_x^{p}(0, \infty))},
&-&\tfrac 12 \leqslant s < \tfrac 12.
\label{nh-W0-est-l} 
}
}
Hence, IBVPs \eqref{nh-lh-ibvp-r} and \eqref{nh-flh-ibvp-r} can be treated as specific cases of the problem
\sss{\label{nh-ibvp-r}
\ddd{
&u_t-u_{xx} = 0, &&x\in (0,\infty), \ t\in (0,T),
\label{nh-ibvp-r-eq} \\
&u(x,0)= 0, &&x\in [0,\infty),  \label{nh-ibvp-r-ic} \\[-.5em]
&u(0,t) = g_0(t)\in H_t^{\frac{2s+1}{4}}(0, T),\quad &&t\in [0, T],
}
}
with $g_0(0)=0$ for $s>\tfrac 12$, whose solution is given by the UTM formula \eqref{nh-flh-utm-sol-T} with $u_0=f=0$.

\subsection{An IBVP with a compactly supported boundary datum}
For the purpose of obtaining Sobolev-type estimates, it would be 
convenient if formula \eqref{nh-flh-utm-sol-T} involved  the \textit{Fourier transform} of $g_0$  instead of the time transform $\widetilde g_0$. Note that this would indeed be the case if $g_0$ were \textit{compactly supported} in $[0, T]$. Recalling that $T<1$, this observation motivates the ``embedding'' of IBVP \eqref{nh-ibvp-r} inside the \textit{pure IBVP}
\sss{\label{nh-ibvp-rv}
\ddd{
&v_t-v_{xx} = 0, &&x\in (0,\infty), \ t\in (0,2),
\label{nh-ibvp-rv-eq} \\
&v(x,0)= 0, &&x\in [0,\infty),  \label{nh-ibvp-rv-ic} \\[-.5em]
&v(0,t) = g(t)\in H_t^{\frac{2s+1}{4}}(\mathbb R), \quad && 
t\in [0, 2],
\label{nh-ibvp-rv-bc}
}
}
where the boundary datum $g\in H_t^{\frac{2s+1}{4}}(\mathbb R)$ is an extension of $g_0$ such that $\text{supp}(g)\subset (0, 2)$.
In particular, if $0 \leqslant s <\frac 12$ then $g$ is simply the extension of $g_0\in H_t^{\frac{2s+1}{4}}(0, T)$ by zero outside $(0, T)$, while if $\frac 12<s\leqslant \frac 32$ then $g$ is defined along the lines of \cite{fhm2017} as
$$
g(t) = \left\{
\def\arraystretch{1}
\begin{array}{ll}
E_\theta(t), &t\in (0,2),
\\
0, &t\in (0,2)^c,
\end{array}
\right.
$$
where 
$
E_\theta(t) 
=
\theta(t) E(t) 
$
with $\theta\in C^\infty_0(\mathbb R)$ a smooth cut-off function  such that $|\theta(t)|\leqslant 1$ for all $t\in\mathbb R$, 
$\theta(t) = 1$ for $|t|\leqslant 1$ and $\theta(t)= 0$ for $|t|\geqslant 2$, and $E\in H_t^{\frac{2s+1}{4}}(\mathbb R)$ an extension of $g_0\in H_t^{\frac{2s+1}{4}}(0, T)$  such that 
$
\no{E}_{H_t^{\frac{2s+1}{4}}(\mathbb R)}
\leqslant
c
\left\| g_0 \right\|_{H_t^{\frac{2s+1}{4}}(0,T)}.
$
In both cases, by construction we have $\text{supp}(g)\subset (0, 2)$ and 
\eee{\label{nh-bcdet}
\no{g}_{H_t^{\frac{2s+1}{4}}(\mathbb R)}
\leqslant 
c_s 
\no{g_0}_{H_t^{\frac{2s+1}{4}}(0, T)},
\quad
0\leqslant s\leqslant \tfrac 32, \ s\neq\tfrac 12.
}

By the definition of $g$ as an extension of $g_0$, the pure IBVP \eqref{nh-ibvp-rv} restricted on $(0, \infty) \times (0, T)$ becomes IBVP \eqref{nh-ibvp-r}.
Therefore, the solution of the latter problem can be estimated via the solution of former one, which is given by the UTM formula \eqref{nh-flh-utm-sol-T}  as
\ddd{\label{nh-ibvp-rv-sol}
v(x,t) 
= 
S\big[0, g; 0\big](x, t) 
=
\frac{1}{\pi} \int_{k=0}^\infty e^{ia^3kx+ik^2 t} \, k \widehat g(k^2) dk
+
\frac{1}{\pi} \int_{k=0}^\infty e^{iakx-ik^2t}  \, k\widehat g(-k^2)dk
}
with $a = e^{i\frac{\pi}{4}}$ and 
$
\widehat g(\tau) = \int_{t\in \mathbb R} e^{-i\tau t} g(t) dt
$
being the Fourier transform of $g$.

\subsection{Sobolev-type estimates}

We begin with the estimation of the pure IBVP \eqref{nh-ibvp-rv} in Sobolev spaces, which reveals to us the correct space for the boundary datum, namely the space $H_t^{\frac{2s+1}{4}}(0, T)$.

\begin{theorem}[\b{Sobolev-type estimates for the pure IBVP}]
\label{nh-ibvp-rv-t}
The solution $v=S\big[0,g;0\big]$ of the pure IBVP \eqref{nh-ibvp-rv} given by the UTM formula \eqref{nh-ibvp-rv-sol} admits the space and time estimates
\ddd{
\no{v}_{C([0, 2]; H_x^s(0,\infty))}
\leqslant
c_s \no{g}_{H_t^{\frac{2s+1}{4}}(\mathbb R)},
\quad && s\geqslant 0,
\label{nh-ibvp-rv-se}
\\
\no{v}_{C([0, \infty); H_t^{\frac{2s+1}{4}}(0,2))}
\leqslant
c_s \no{g}_{H_t^{\frac{2s+1}{4}}(\mathbb R)},
\quad && s \in \mathbb R.
\label{nh-ibvp-rv-te}
}
\end{theorem}

\begin{remark}
The space estimate \eqref{nh-ibvp-rv-se} for the pure IBVP \eqref{nh-ibvp-rv} is central in our analysis as it motivates the space $H_t^{\frac{2s+1}{4}}(0, T)$ for the boundary datum $g_0$ of the nonlinear IBVP \eqref{nh-rd-ibvp}. 
Recall that another source of motivation for the boundary data space is the time regularity \eqref{nh-ivp-te} of the homogeneous linear IVP solution.
\end{remark}

\noindent
\begin{Proof}[Proof of Theorem \ref{nh-ibvp-rv-t}]
To establish the space estimate \eqref{nh-ibvp-rv-se}, we write $v  = v_1 + v_2$ with
\ddd{
\label{v1-def}
&v_1(x, t)
=
\int_{k=0}^\infty
e^{i \gamma_1 k x} G_1(k, t) d k, \quad
G_1(k, t)
=
\frac{1}{\pi}e^{i k^2 t } k \widehat g(k^2),
\quad
\gamma_1= a^3 = e^{i\frac{3\pi}{4}},
\\
\label{v2-def}
&v_2(x, t)
=
\int_{k=0}^\infty
e^{i \gamma_2 k x} G_2(k, t) d k, \quad
G_2(k, t)
=
\frac{1}{\pi}e^{-i k^2 t } k \widehat g (-k^2),
\quad
 \gamma_2= a = e^{i\frac{\pi}{4}}.
}
The estimation of $v_1$ and $v_2$ is entirely analogous. Thus, we only give the details for $v_1$.
We employ the physical space definition of the $H_x^s(0,\infty)$-norm:
\eee{
\label{nls-hl-sobfrac}
\no{v_1(t)}_{H_x^s(0, \infty)}
=
\sum_{j=0}^{\left\lfloor s \right\rfloor}
\no{\p_x^j v_1(t)}_{L_x^2(0, \infty)}
+
\big\| \p_x^{\left\lfloor s \right\rfloor}v_1(t)\big\|_\beta,\quad s=\left\lfloor s \right\rfloor +\beta\geqslant 0, \ 0\leqslant \beta <1, 
}
where $\lfloor \cdot \rfloor$ denotes the floor function
and   the fractional norm $\no{\cdot}_\beta$ is defined by 
\eee{
\no{v_1(t)}_{\beta}
=
\left(\int_{x=0}^\infty \int_{z=0}^\infty \dfrac{\left| v_1(x+z, t)- v_1(x, t)\right|^2}{z^{1+2\beta}}\, dzdx\right)^{\frac 12}, \quad 0 < \beta < 1.
\nn
}
There are three cases to consider:
(i) $\left\lfloor s \right\rfloor=0$ and $\beta\neq 0$;
(ii) $\beta=0$; 
(iii) $\left\lfloor s \right\rfloor\neq 0$ and  $\beta\neq 0$. 
\vskip 3mm
\noindent
\textit{\b{(i) The case $\left\lfloor s \right\rfloor =0$ and $\beta\neq 0$.}}
Then,  $s= \beta\in (0,1)$ and we need to estimate  $\no{v_1(t)}_{L_x^2(0,\infty)}$ and $\no{v_1(t)}_\beta$. The first norm will be estimated together with the $L_x^2(0, \infty)$-norms of the higher derivatives of $v_1$ in case (ii). 
For the second norm, we have
$$
\no{v_1(t)}_\beta^2
\leqslant
\int_{z=0}^\infty\int_{x=0}^\infty \frac{1}{z^{1+2\beta}}
\left(\int_{k=0}^\infty
\left|e^{i\gamma_1 k (x+z)}-e^{i\gamma_1 k x}\right| \left|G_1(k ,t)\right|dk \right)^2
dx dz
$$
and use the following lemma.
\begin{lemma}[\cite{fhm2016}, Lemma 2.1]
\label{nh-gamma-l}
If $\gamma=\gamma_R+i\gamma_I$ with $\gamma_I>0$, then 
$$
\left|e^{i\gamma  k x}-e^{i\gamma k z}\right|
\leqslant
\sqrt 2\left(1+\tfrac{|\gamma_R|}{\gamma_I}\right)
\left|e^{-\gamma_I  k x}-e^{-\gamma_I  k z}\right|
\quad
\forall k, x, z \geqslant 0.
$$
\end{lemma}

Employing Lemma \ref{nh-gamma-l} with $\gamma= \gamma_1 = e^{i\frac{3\pi}{4}}$ and subsequently making the change of variables
 $\frac{\sqrt 2}{2}x \to x$, $\frac{\sqrt 2}{2}z \to z$, we obtain
$$
\label{v1-bound-L0}
\no{v_1(t)}_\beta^2
\lesssim
\int_{z=0}^\infty
\frac{1}{z^{1+2\beta}}
\int_{x=0}^\infty
\left(\int_{k=0}^\infty e^{-k x}\big(1-e^{-k z}\big)\left|G_1(k,t)\right|dk \right)^2
dx dz.
$$
We identify the  $k$-integral in \eqref{v1-bound-L0} as the Laplace transform of $Q_{z,t}(k) \doteq \big(1-e^{-k z}\big)\left|G_1(k,t)\right|$. Therefore, using Lemma \ref{nh-tlap-l}  we infer
\ddd{
\no{v_1(t)}_\beta^2
&\lesssim
\int_{z=0}^\infty \frac{1}{z^{1+2\beta}} \no{Q_{z,t}}^2_{L^2_{k}(0,\infty)}dz
\nonumber
\\
&=
\int_{k=0}^\infty \left|G_1(k,t)\right|^2\int_{z=0}^\infty \frac{\big(1-e^{-k z}\big)^2}{z^{1+2\beta}}dz dk
\lesssim
\int_{k=0}^\infty k^2 \big|\widehat g(k^2)\big|^2 |k|^{2\beta}dk,
\nn
}
with the last inequality due to estimate \eqref{nh-gamma-est} and the definition \eqref{v1-def} of $G_1$. 
Then, making the change of variable $\tau=k^2$ we obtain 
\eee{\label{nh-v1-se-1}
\no{v_1(t)}_\beta^2
\lesssim
\int_{\tau=0}^\infty |\tau|^{\frac{2\beta+1}{2}} \big|\widehat g(\tau)\big|^2 d\tau
\leqslant
\no{g}_{H_{t}^{\frac{2\beta+1}{4}}(\mathbb R)}, \quad 0<\beta<1,
}
which is the desired estimate for $v_1$ and which indicates that the \textit{optimal} value of $m$ is indeed $\frac{2s+1}{4}$.
\vskip 3mm
\noindent
\textit{\b{(ii) The case $\beta=0$.}}
Differentiating formula \eqref{v1-def} and taking the $L^2$-norm, we have
$$
\no{\p_x^jv_1(t)}_{L_x^2(0,\infty)}
\leqslant
\no{\int_{k=0}^\infty e^{-\frac{\sqrt 2}{2} kx} \,|k|^{j+1} \,  \big|\widehat g(k^2)\big| dk}_{L_x^2(0, \infty)}.
$$
Therefore, exploiting once again the boundedness of the Laplace transform in $L^2$  and estimating as in case (i) above, we find
\eee{
\no{\p_x^jv_1(t)}_{L_x^2(0,\infty)}
\lesssim
 \big\| |k|^{j+1} \, \widehat g(k^2)\big\|_{L_k^2(0,\infty)}
\lesssim
\no{g}_{H_t^{\frac{2j+1}{4}}(\mathbb R)}^2,
 \quad
 j=0, 1, \ldots, \left\lfloor s \right\rfloor,
 \nn
}
so that
\eee{\label{nh-v1-se-2}
\no{v_1(t)}_{H_x^{\lfloor s \rfloor}(0,\infty)}
\lesssim
\no{g}_{H_t^{\frac{2\lfloor s \rfloor +1}{4}}(\mathbb R)}, \quad \lfloor s \rfloor \geqslant 0.
}
\vskip 3mm
\noindent
\textit{\b{(iii) The case $\left\lfloor s \right\rfloor \neq 0$ and $\beta \neq 0$.}}
We now have $s=\left\lfloor s \right\rfloor+\beta$
with $\left\lfloor s \right\rfloor\geqslant 1$ and $0<\beta<1$.
Thanks to our earlier work in cases (i) and (ii), it suffices to estimate the fractional norm
$\big\|\p_x^{\left\lfloor s \right\rfloor}v_{1}(t)\big\|_{\beta}$.
Proceeding as in case (i), we obtain 
\eee{
\big\| \p_x^{\left\lfloor s \right\rfloor}v_{1}(t)\big\|_{\beta}^2
\lesssim
\int_{k=0}^\infty    k^{2(\left\lfloor s \right\rfloor + \beta)+2} \big|\widehat g(k^2)\big|^2 dk
=
\no{k^{\left\lfloor s \right\rfloor + \beta+1} \widehat g(k^2)}_{L_k^2(0, \infty)}^2,
\nn
}
which can be estimated  like the corresponding term in case (ii)  to yield
\eee{\label{nh-v1-se-3}
\big\|\p_x^{\left\lfloor s \right\rfloor}v_{1}(t)\big\|_\beta
\lesssim
\no{g}_{H_t^{\frac{2s+1}{4}}(\mathbb R)}, \quad \lfloor s \rfloor \geqslant 1.
}

Estimates \eqref{nh-v1-se-1}, \eqref{nh-v1-se-2} and \eqref{nh-v1-se-3} combined with the definition \eqref{nls-hl-sobfrac} imply
\eee{\label{nh-v1-se}
\no{v_1(t)}_{H_x^s(0, \infty)}
\leqslant
c_s \no{g}_{H_t^{\frac{2s+1}{4}}(\mathbb R)}, \quad s\geqslant 0, \ t\in [0, 2],
}
which is the space estimate \eqref{nh-ibvp-rv-se} for $v_1$. As noted earlier, this estimate can be established for $v_2$ in the exact same way. Hence, the proof of the space estimate \eqref{nh-ibvp-rv-se} for $v$ is complete. Continuity in time follows via the dominated convergence theorem as in Theorem 5 of \cite{fhm2017}.

\vskip 3mm
The time estimate \eqref{nh-ibvp-rv-te} is easier to establish. In particular, the change of variable $k=\sqrt{\tau}$ turns formula \eqref{nh-ibvp-rv-sol} into
$$
v(x,t)
\simeq
\int_{\tau=0}^\infty e^{i a^3 \sqrt \tau x+i\tau t}  \, \widehat g(\tau)  d\tau
+
\int_{\tau=-\infty}^0 e^{ia\sqrt{-\tau} x+ i\tau t} \, \widehat g(\tau)  d\tau.
$$
Therefore, since the $x$-exponentials involved above are bounded by $1$ for $x\geqslant 0$, we obtain
\ddd{
\no{v(x)}_{H_t^{\frac{2s+1}{4}}(\mathbb R)}^2
&\lesssim
%\no{
%\int_{\tau=0}^\infty e^{i a^3 \sqrt \tau x+i\tau t}  \, \widehat g(\tau)  d\tau}_{H_t^{\frac{2s+1}{4}}(\mathbb R)}^2
%+
%\no{\int_{\tau=-\infty}^0 e^{ia\sqrt{-\tau} x+ i\tau t} \, \widehat g(\tau)  d\tau}_{H_t^{\frac{2s+1}{4}}(\mathbb R)}^2
%\nn\\
%&=
\int_{\tau=0}^\infty \left(1+\tau^2\right)^{\frac{2s+1}{4}}
e^{-\sqrt 2 \sqrt \tau x}  \big|\widehat g(\tau) \big|^2 d\tau
+
\int_{\tau=-\infty}^0 \left(1+\tau^2\right)^{\frac{2s+1}{4}} e^{-\sqrt 2 \sqrt{-\tau} x} \big| \widehat g(\tau) \big|^2 d\tau
\nn\\
&\leqslant
\int_{\tau=0}^\infty \left(1+\tau^2\right)^{\frac{2s+1}{4}}
 \big|\widehat g(\tau) \big|^2 d\tau
+
\int_{\tau=-\infty}^0 \left(1+\tau^2\right)^{\frac{2s+1}{4}}
 \big| \widehat g(\tau) \big|^2 d\tau
=
\no{g}_{H_t^{\frac{2s+1}{4}}(\mathbb R)}^2.
\nn
}
Continuity in $x$ follows via the dominated convergence theorem as in Theorem 5 of \cite{fhm2017}.
\end{Proof}

\subsection{Additional estimate}
We complete the analysis of IBVP \eqref{nh-ibvp-rv} with the estimation in the space $C^\alpha([0, T]; L_x^{p}(0, \infty))$.
\begin{theorem}[\b{$C^\alpha([0, T]; L_x^{p}(0, \infty))$-estimate for the pure IBVP}]
\label{nh-ibvp-rv-ca-t}
If $\frac 12-\frac{1}{p}<s<\frac 12$, then  the solution $v = S\big[0, g; 0\big]$ of the pure IBVP \eqref{nh-ibvp-rv}  given by the UTM formula \eqref{nh-ibvp-rv-sol} admits the estimate
\eee{\label{nh-ibvp-rv-ca-est-0}
\no{v}_{C^\alpha([0, T]; L_x^{p}(0, \infty))}
\lesssim 
T^\alpha  \no{g}_{H_t^{\frac{2s+1}{4}}(\mathbb R)}.
}
\end{theorem}

\begin{proof}[\textnormal{\textbf{Proof of Theorem \ref{nh-ibvp-rv-ca-t}.}}]
We decompose $v=v_1+v_2$ as in the proof of Theorem \ref{nh-ibvp-rv-t} and estimate $v_1$ and $v_2$ separately. Using Minkowski's integral inequality and the fact that $x\geqslant 0$, we find
\ddd{
\no{v_1(t)}_{L_x^{p}(0, \infty)}
&\lesssim
\int_{k=0}^\infty \left(\int_{x=0}^\infty \left|e^{i a^3 k x + i k^2 t}   k \widehat g  (k^2)\right|^{p}  dx \right)^{\frac{1}{p}} dk
\nn\\
&=
\int_{k=0}^\infty 
k \big|\widehat g  (k^2)\big|
\left(\int_{x=0}^\infty  e^{-\frac{1}{\sqrt 2} p  k x}    dx \right)^{\frac{1}{p}} 
dk
\simeq
\int_{k=0}^\infty 
k^{1-\frac{1}{p}} \big|\widehat g  (k^2)\big|  dk.
\nn
}
Thus, letting $ k = \sqrt \tau$ and applying the Cauchy-Schwarz inequality in $\tau$, we obtain
\eee{
\no{v_1(t)}_{L_x^{p}(0, \infty)}
\lesssim
\left(
\int_{\tau=0}^\infty 
\tau^{-\frac{1}{p}} \left(1+\tau^2\right)^{-\frac{2s+1}{4}} d\tau\right)^{\frac 12} 
\no{g}_{H_t^{\frac{2s+1}{4}}(\mathbb R)}
\simeq
\no{g}_{H_t^{\frac{2s+1}{4}}(\mathbb R)}, \quad s>\tfrac 12-\tfrac{1}{p},
\nn
}
from which we deduce estimate \eqref{nh-ibvp-rv-ca-est-0} for $v_1$. 
The estimation of $v_2$ is entirely analogous.
Note that the restriction $s<\frac 12$ appearing in the hypothesis of the theorem originates from the condition on $b$ involved in \eqref{nh-cbeta-def} and is not needed in the present proof.
\end{proof}

\subsection{Summary of linear IBVP estimates on the half-line}
In view of Theorems \ref{nh-ivp-t}%, \ref{nh-fivp-t},  \ref{nh-ivp-ca-t},
-\ref{nh-fivp-ca-t}, \ref{nh-ibvp-rv-t}, \ref{nh-ibvp-rv-ca-t} and inequalities \eqref{nh-uU}%\eqref{nh-Ff-se}, 
-\eqref{nh-Ff-ca-est}, \eqref{nh-G0-W0-est}, \eqref{nh-bcdet}, the decomposition \eqref{nh-flh-ibvp-sup} yields the following estimates for the forced linear IBVP \eqref{nh-flh-ibvp}.

\begin{theorem}[\b{Estimates for the forced linear heat on the half-line}] 
\label{nh-flh-ibvp-wp-t}
For $f=\prod_{j=1}^p  f_j$, the solution $u=S\big[u_0, g_0; f\big]$ of the forced linear heat IBVP  \eqref{nh-flh-ibvp}  satisfies the space estimates
\sss{\label{nh-flh-ibvp-se}
\ddd{
\no{S\big[u_0, g_0; f\big]}_{C([0, T]; H_x^s(0, \infty))}
&\leqslant
c_{s, p}\Big(
\no{u_0}_{H_x^s(0, \infty)}
+
\no{g_0}_{H^{\frac{2s+1}{4}}_t(0, T)}
\label{nh-flh-ibvp-se-h}
\\
&\hskip 1.4cm
+
\sqrt T \prod_{j=1}^p  \no{f_j}_{C([0, T]; H_x^s(0, \infty))}\Big),
\quad
\tfrac 12<s<\tfrac 32,
\nn\\
\no{S\big[u_0, g_0; f\big]}_{C([0, T]; H_x^s(0, \infty))}
&\leqslant
c_{s, p}\Big(
\no{u_0}_{H_x^s(0, \infty)}
+
\no{g_0}_{H^{\frac{2s+1}{4}}_t(0, T)}
\label{nh-flh-ibvp-se-l}
\\
&\hskip 1.4cm
+
\sqrt T \prod_{j=1}^p  \no{f_j}_{C^\alpha([0, T]; L_x^{p}(0, \infty))}\Big),
\quad
0\leqslant s<\tfrac 12,
\nn
}
}
the time estimates
\sss{\label{nh-flh-ibvp-te}
\ddd{
\no{S\big[u_0, g_0; f\big]}_{C([0, \infty); H_t^{\frac{2s+1}{4}}(0, T))}
&\leqslant
c_{s, p}\Big(
\no{u_0}_{H_x^s(0, \infty)}
+
\no{g_0}_{H^{\frac{2s+1}{4}}_t(0, T)}
\label{nh-flh-ibvp-te-h}
\\
&\hskip 1.3cm
+
\sqrt T \prod_{j=1}^p  \no{f_j}_{C([0, T]; H_x^s(0, \infty))}\Big),
\quad
\tfrac 12<s<\tfrac 32, 
\nn\\
\no{S\big[u_0, g_0; f\big]}_{C([0, \infty); H_t^{\frac{2s+1}{4}}(0, T))}
&\leqslant
c_{s, p}\Big(
\no{u_0}_{H_x^s(0, \infty)}
+
\no{g_0}_{H^{\frac{2s+1}{4}}_t(0, T)}
\label{nh-flh-ibvp-te-l}
\\
&\hskip 1.3cm
+
\sqrt T \prod_{j=1}^p  \no{f_j}_{C^\alpha([0, T]; L_x^{p}(0, \infty))}\Big),
\quad
0\leqslant s<\tfrac 12,
\nn
}
}
and the $L^{p}$-estimate
\ddd{\label{nh-flh-ibvp-ca-est}
\no{S\big[u_0, g_0; f\big]}_{C^\alpha([0, T]; L_x^{p}(0, \infty))}
&\leqslant
c_{s, p}\Big(
\no{u_0}_{H_x^s(0, \infty)}
+
\no{g_0}_{H^{\frac{2s+1}{4}}_t(0, T)}
\\
&\hskip 1.3cm
+
T^{\alpha}
\prod_{j=1}^p \no{f_j}_{C^{\alpha}([0, T]; L_x^{p}(0, \infty))}\Big),
\quad
\tfrac 12-\tfrac{1}{p} < s<\tfrac 12, 
\nn
}
where $c_{s, p}>0$ is a constant that depends only on $s$ and $p$.
\end{theorem}

%
%
%
%
%%%%%%%%%%%%%%%%%%%%%%%%% 
%
%     Local Well-Posedness on the Half-Line
%
%%%%%%%%%%%%%%%%%%%%%%%%%  
%
%
%
\section{Local Well-Posedness on the Half-Line}
\label{nh-lwp-s}

The linear estimates of Theorem \ref{nh-flh-ibvp-wp-t} for the forced linear heat IBVP \eqref{nh-flh-ibvp} will now be employed for showing well-posedness via contraction mapping of the (nonlinear) reaction-diffusion IBVP \eqref{nh-rd-ibvp}. We shall begin with the case of ``smooth'' data ($s>\frac 12$), which corresponds to Theorem \ref{nh-ibvp-wp-h-t}, and then proceed to the case of ``rough'' data ($s<\frac 12$), which corresponds to Theorem \ref{nh-ibvp-wp-l-t}.

\subsection{Proof of Theorem \ref{nh-ibvp-wp-h-t}}
\label{nh-nh-ibvp-wp-h-t-ss}
We shall establish local well-posedness in the sense of Hadamard, i.e. we shall show existence, uniqueness and continuous dependence of the solution on the initial and boundary data.
\vskip 3mm
\noindent
\textbf{I. Existence and uniqueness.}
For $f = |u|^{p-1}u=u^p$, $\tfrac{p-1}{2}\in\mathbb N$, the forced linear IBVP solution $S\big[u_0, g_0; f\big]$  given by \eqref{nh-flh-utm-sol-T} induces the following  iteration map for the reaction-diffusion IBVP \eqref{nh-rd-ibvp}:
\eee{\nn
u\mapsto \Phi  u = \Phi_{u_0, g_0}  u
 \doteq
S\big[u_0, g_0; u^p \big].
}
Existence and uniqueness  in the case of ``smooth'' data ($s>\frac 12$) will be established by showing that the above map is a contraction in the space 
\eee{\label{nh-x-def-p}
X
=
C([0, T^*]; H_x^s(0, \infty))\cap C([0, \infty); H_t^{\frac{2s+1}{4}}(0, T^*)),
}
where the lifespan $T^*\in (0, T]$ is to be determined.
\vskip 3mm
\noindent
\textit{\b{Showing that $u\mapsto \Phi  u$ is onto $X$.}}
Let $B(0, \varrho)=\left\{u\in X : \no{u}_X\leqslant \varrho\right\}$ be a ball centered at $0$ with radius $\varrho = 2c_{s, p}\no{(u_0,g_0)}_D$, where $c_{s, p}>0$ is the constant appearing in Theorem \ref{nh-flh-ibvp-wp-t} and the data norm $\no{\cdot}_D$ is defined by
\eee{\label{nh-data-norm-def}
\no{(u_0,g_0)}_D
=
\no{u_0}_{H_x^s(0,\infty)}+\no{g_0}_{H_t^{\frac{2s+1}{4}}(0,T)}.
}
For $u\in B(0, \varrho)$ and $\tfrac 12 < s < \tfrac 32$, the space estimate  \eqref{nh-flh-ibvp-se-h}
and the time estimate   \eqref{nh-flh-ibvp-te-h} imply
\ddd{
\no{\Phi  u}_{X} 
&=
\sup_{t\in[0, T^*]} 
\no{S\big[u_0, g_0; u^p \big](t)}_{H_x^s(0,\infty)} 
+
\sup_{x\in[0, \infty)} 
\no{S\big[u_0, g_0; u^p \big](x)}_{H_t^{\frac{2s+1}{4}}(0, T^*)} 
\nn\\
&\leqslant
c_{s, p} 
\left(
\no{\left(u_0, g_0\right)}_D 
+
\sqrt{T^*}
\left\|u \right\|_X^p
\right)
\leqslant
\frac \varrho2  + c_{s, p}    \sqrt{T^*} \, \varrho^p.
\nn
}
Thus, for $T^*\in (0, T]$  such that
\eee{\label{nh-lifespan-onto}
\frac \varrho2+c_{s, p}   \sqrt{T^*}\, \varrho^p \leqslant  \varrho
  \Leftrightarrow 
T^*
\leqslant
\frac{1}{(2c_{s, p} )^{2p}  \left\|\left(u_0, g_0\right)\right\|_D^{2(p-1)}}
}
the map $u\mapsto \Phi u$ is onto the ball $B(0, \varrho)$.
\vskip 3mm
\noindent
\textit{\b{Showing that  $u\mapsto \Phi  u$ is a contraction in $X$.}}
We shall show that 
\eee{
\label{nh-contr-ineq}
\left\| \Phi  u_1 - \Phi  u_2 \right\|_{X} \leqslant \tfrac 12 \left\| u_1-u_2\right\|_X \quad \forall u_1, u_2\in B(0, \varrho)\subset X.
}
Noting that
$
\Phi  u_1 - \Phi  u_2
=
S\big[0, 0;  u_1^p-u_2^p \, \big]
$
and using estimates \eqref{nh-flh-ibvp-se-h} and \eqref{nh-flh-ibvp-te-h}, we obtain
$$
\left\| \Phi  u_1 - \Phi  u_2 \right\|_{X} 
\leqslant
c_{s, p} \, \sqrt{T^*} \sup_{t\in[0,T^*] }  \left\| \left(u_1^p-u_2^p\right)(t) \right\|_{H_x^s(0,\infty)},
\quad \tfrac 12 < s < \tfrac 32.
$$
Thus, employing the identity
\eee{
u_1^p  -  u_2^p
=
\big(u_1^{p-1} + u_1^{p-2} u_2 + \ldots + u_1 u_2^{p-2} + u_2^{p-1}\big) \left(u_1-u_2\right)
\label{nn-diff}
}
and the algebra property in $H_x^s(0, \infty)$,   we deduce  
$$
\left\| \Phi  u_1 - \Phi  u_2 \right\|_{X} 
\leqslant
p\, c_{s, p}  \sqrt{T^*}  \varrho^{p-1} \left\| u_1 - u_2\right\|_X, \quad 
\tfrac 12 < s < \tfrac 32.
$$
Hence, for $T^*$ such that
\eee{\label{nh-lifespan-contr}
p\, c_{s, p} \sqrt{T^*} \varrho^{p-1} 
\leqslant \frac 12
\Leftrightarrow
T^*  \leqslant \frac{1}{p^2 \left(2c_{s, p} \right)^{2p}  \no{(u_0, g_0)}_D^{2\left(p-1\right)}}
}
the contraction inequality \eqref{nh-contr-ineq} is satisfied.
Since $p>1$,  condition \eqref{nh-lifespan-onto} is weaker than condition \eqref{nh-lifespan-contr}. Therefore, for lifespan $T^*$ given by
\eee{\label{nh-lifespan}
T^*
=
\min
\bigg\{ 
T,  \, 
\frac{1}{p^2 \left(2c_{s, p}\right)^{2p} \no{(u_0, g_0)}_D^{2\left(p-1\right)}}
\bigg\}, \quad c_{s, p}>0,
}
the contraction mapping theorem implies that the map $u\mapsto\Phi u$ has a unique fixed point in $B(0, \varrho)$. Equivalently, the integral equation $u=\Phi  u$ for the reaction-diffusion IBVP \eqref{nh-rd-ibvp} has a unique solution $u\in B(0, \varrho)\subset X$.
The proof of existence and uniqueness is complete.
\vskip 3mm
\noindent
\textbf{II. Continuity of the data-to-solution map.} 
We complete Hadamard well-posedness by showing that the data-to-solution map
$
H_x^s(0,\infty)\times H_t^{\frac{2s+1}{4}}(0,T) 
\ni
\left(u_0, g_0\right) 
\mapsto 
u\in X
$
is continuous.

Let $\big(u_{0_1}, g_{0_1}\big)$ and $\big(u_{0_2}, g_{0_2}\big)$ be two pairs of data in the data space $D$ that lie inside a ball $B(\varrho, \delta)\subset D$ of radius $\delta>0$ centered at a distance $\varrho$ from the origin. 
Denote by $u_1=\Phi_{u_{0_1},\, g_{0_1}} u_1$ and $u_2 = \Phi_{u_{0_2},\, g_{0_2}} u_2$  the  corresponding solutions to the reaction-diffusion IBVP \eqref{nh-rd-ibvp}, and by $T_{u_1}$ and $T_{u_2}$ the lifespans of those solutions  given according to \eqref{nh-lifespan}.
Since 
$
\max\left\{\no{\left(u_{0_1}, g_{0_1}\right)}_D, \no{\left(u_{0_2}, g_{0_2}\right)}_D\right\} \leqslant \varrho+\delta
$
and $p>1$,  it follows that
$$
\min\left\{T_{u_1}, T_{u_2}\right\}
 \geqslant  
 \min\bigg\{ T, \, \frac{1}{p^2 \left(2c_{s, p} \right)^{2p}  \left(\varrho+\delta\right)^{2\left(p-1\right)}}\bigg\}   \doteq T_c.
$$
Hence, both $u_1$ and $u_2$ are guaranteed to exist for $0\leqslant t\leqslant T_c$. Replacing $T^*$ with the common lifespan $T_c$ in  \eqref{nh-x-def-p}  gives rise to the space  
\eee{\nn
X_c = C([0, T_c]; H_x^s(0,\infty)) \cap C([0, \infty); H_t^{\frac{2s+1}{4}}(0, T_c)).
}
Observe that $X_{u_1}, X_{u_2} \subset X_c$ with the spaces $X_{u_1}$ and $X_{u_2}$ defined by \eqref{nh-x-def-p} with $T^*$ replaced by
$T_{u_1}$ and $T_{u_2}$ respectively.
We shall now determine the radius $\varrho_c$ of a ball $B(0, \varrho_c)\subset X_c$ such that $u_1, u_2\in B(0, \varrho_c)$ and  
\eee{
\label{nls-hl-lipaim}
\no{u_1 - u_2}_{X_c} \leqslant 2c_{s, p}  \left\| \left(u_{0_1}, g_{0_1}\right) - \left(u_{0_2}, g_{0_2}\right)\right\|_D.
}
Since $u_1$ and $u_2$ are obtained as fixed points of the maps $u_1\mapsto \Phi u_1$ and $u_2\mapsto \Phi u_2$ in   $X_{u_1}$ and $X_{u_2}$ respectively, we have
$$
\no{u_1 - u_2}_{X_c}
=
\left\| \Phi u_1 - \Phi u_2 \right\|_{X_c}
=
\left\| 
S\big[u_{0_1}-u_{0_2}, g_{0_1}-g_{0_2}; u_1^p-u_2^p \big]
\right\|_{X_c}.
$$
Thus, estimates \eqref{nh-flh-ibvp-se-h} and \eqref{nh-flh-ibvp-te-h} together with identity \eqref{nn-diff} and the algebra property in $H_x^s(0, \infty)$  imply
\eee{\label{nn-lip-cond}
\no{u_1 - u_2}_{X_c} 
\leqslant 
\frac{c_{s, p} }{1- c_{s, p}   \sqrt{T_c}\, p\, \varrho_c^{p-1}}  
\left\| \left(u_{0_1}, g_{0_1}\right) - \left(u_{0_2}, g_{0_2}\right)\right\|_D
}
provided that
$c_{s, p}  \sqrt{T_c}\, p \, \varrho_c^{p-1} < 1$, 
which is satisfied for
$\varrho_c = \left(2c_{s, p}   \sqrt{T_c}\, p\right)^{-\frac{1}{p-1}}$.
For this choice of $\varrho_c$,  inequality \eqref{nn-lip-cond} becomes the Lipschitz condition \eqref{nls-hl-lipaim}. The proof of Theorem \ref{nh-ibvp-wp-h-t} is complete.

\subsection{Proof of Theorem \ref{nh-ibvp-wp-l-t}}
\noindent
\textbf{I. Existence and uniqueness.}
For $s<\frac 12$ (``rough'' data), we refine the solution space $X$ defined by  \eqref{nh-x-def-p} by intersecting it with the space $C^\alpha([0, T^*]; L_x^{p}(0, \infty))$, i.e. we replace $X$ by
\eee{\label{nh-y-def-p}
Y
=
C([0, T^*]; H_x^s(0, \infty))
\cap 
C([0, \infty); H_t^{\frac{2s+1}{4}}(0, T^*))
\cap
C^\alpha([0, T^*]; L_x^{p}(0, \infty)),
}
where the lifespan $T^*\in (0, T]$ is to be determined.
In analogy to the case $s>\frac 12$ of ``smooth'' data, existence and uniqueness of solution to the reaction-diffusion IBVP \eqref{nh-rd-ibvp}  will be established by showing that the iteration map
\eee{\nn
u\mapsto \Phi  u = \Phi_{u_0, g_0}  u
 \doteq
S\big[u_0, g_0; |u|^{p-1}u \big]
}
is a contraction in $Y$. 
Note that since for $s<\frac 12$ the forcing $f=|u|^{p-1}u$ appears in the linear estimates of Theorem \ref{nh-flh-ibvp-wp-t} under the  $L_x^{p}(0, \infty)$-norm instead of the $H_x^s(0, \infty)$-norm, the restriction $\frac{p-1}{2}\in\mathbb N$ that was present for $s>\frac 12$ is not needed anymore, i.e.  $p=2, 3, 4,\ldots.$
\vskip 3mm
\noindent
\textit{\b{Showing that  $u\mapsto \Phi  u$ is onto $Y$.}}
Let $B(0,\varrho)=\left\{u\in Y : \no{u}_Y\leqslant \varrho\right\}$ be a ball centered at $0$ with radius $\varrho = 2c_{s, p}\no{(u_0,g_0)}_D$, where $c_{s, p}>0$ is the constant appearing in Theorem \ref{nh-flh-ibvp-wp-t} and the norm $\no{\cdot}_D$ is defined by \eqref{nh-data-norm-def}.
For $u\in B(0, \varrho)$ and $\tfrac 12 - \tfrac{1}{p} < s<\tfrac 12$, using the space estimate  \eqref{nh-flh-ibvp-se-l}, the time estimate   \eqref{nh-flh-ibvp-te-l} and the $L^{p}$-estimate \eqref{nh-flh-ibvp-ca-est} with $f_j=u$, we have
\ddd{
\no{\Phi  u}_{Y} 
&=
\sup_{t\in[0, T^*]} 
\no{S\big[u_0, g_0; |u|^{p-1}u \big](t)}_{H_x^s(0,\infty)} 
+
\sup_{x\in[0, \infty)} 
\no{S\big[u_0, g_0; |u|^{p-1}u \big](x)}_{H_t^{\frac{2s+1}{4}}(0, T^*)} 
\nn\\
&\quad
+
 \no{S\big[u_0, g_0; |u|^{p-1}u \big]}_{C^\alpha([0, T^*]; L_x^{p}(0, \infty))} 
\nn\\
&\leqslant
c_{s, p}\,
\Big( 
\no{\left(u_0, g_0\right)}_D 
+
(T^*)^\alpha  \no{u}_{C^\alpha([0, T^*]; L_x^{p}(0, \infty))}^p
\Big)
\leqslant
\frac \varrho2  + c_{s, p} \left(T^*\right)^\alpha \varrho^p,
\nn
}
where we have also used the fact that  $\sqrt{T^*}<(T^*)^\alpha$ since $T^*\leqslant T <1$ and $\alpha<\tfrac 12$. Hence for $T^*\in (0, T]$ satisfying
\eee{\label{nh-lifespan-onto-y}
\frac \varrho2+c_{s, p}   \left(T^*\right)^\alpha \varrho^p \leqslant  \varrho
 \Leftrightarrow 
T^*
\leqslant
\frac{1}{\left(2c_{s, p}\right)^{\frac{p}{\alpha}}  \left\|\left(u_0, g_0\right)\right\|_D^{\frac{p-1}{\alpha}}}
}
the map $u\mapsto \Phi u$ is onto the ball $B(0, \varrho)$.
\vskip 3mm
\noindent
\textit{\b{Showing that $u\mapsto \Phi  u$ is a contraction in $Y$.}}
We shall show that 
\eee{
\label{nh-contr-ineq-y}
\left\| \Phi  u_1 - \Phi  u_2 \right\|_{Y} \leqslant \tfrac 12 \left\| u_1-u_2\right\|_Y \quad \forall u_1, u_2\in B(0, \varrho)\subset Y.
}
\begin{lemma}\label{nn-complex-l}
For  any  $v, w \in\mathbb C$ and any $p\geqslant 1$, we have
\eee{\nn
\big| |v|^{p-1}v-|w|^{p-1}w \big|
\leqslant 
c_{p} \left(|v|^{p-1}+|w|^{p-1}\right) \left|v-w\right|,
\quad
c_p= 2^{p+1}p > 0.
}
\end{lemma}
The proof of Lemma \ref{nn-complex-l} is given at the end of the section. Employing this lemma with $v=u_1$ and $w=u_2$ and then applying H\"older's inequality, we find
\eee{
\no{|u_1|^{p-1}u_1(t) -  |u_2|^{p-1}u_2(t)}_{L_x^1(0, \infty)}
\leqslant
c_p
\left(\no{u_1(t)}_{L_x^{p}(0, \infty)}^{p-1}
+
\no{u_2(t)}_{L_x^{p}(0, \infty)}^{p-1}
\right)
\no{u_1(t)-u_2(t)}_{L_x^{p}(0, \infty)}.
\nn
%\label{nh-diff-y}
}
Combining this inequality with estimates \eqref{nh-flh-ibvp-se-l}, \eqref{nh-flh-ibvp-te-l} and \eqref{nh-flh-ibvp-ca-est}, we obtain
\eee{
\left\| \Phi  u_1 - \Phi  u_2 \right\|_{Y} 
\leqslant
 2^{p+2}p \, c_{s, p} \left(T^*\right)^\alpha \varrho^{p-1} \no{u_1-u_2}_{Y}.
\nn
}
Hence, for $T^*\in (0, T]$ such that
\eee{\label{nh-lifespan-contr-y}
2^{p+2}p \, c_{s, p} \left(T^*\right)^\alpha \varrho^{p-1} \leqslant \frac 12
\Leftrightarrow 
T^* 
\leqslant 
\frac{1}{\left(2^{p+2}p\right)^{\frac{1}{\alpha}} 
\left(2c_{s, p}\right)^{\frac{p}{\alpha}} \no{(u_0, g_0)}_D^{\frac{p-1}{\alpha}}}
}
the contraction inequality \eqref{nh-contr-ineq-y} is satisfied.

Overall, for lifespan $T^*$ given by
\eee{
T^*
=
\min\bigg\{T,\,   \frac{1}{\left(2^{p+2}p\right)^{\frac{1}{\alpha}} 
\left(2c_{s, p}\right)^{\frac{p}{\alpha}}  \no{(u_0, g_0)}_D^{\frac{p-1}{\alpha}}}
\bigg\}, \quad c_{s, p}>0,
\nn
}
the Banach fixed point theorem implies a unique solution $u\in B(0, \varrho)\subset Y$ for the reaction-diffusion IBVP \eqref{nh-rd-ibvp}. The proof of existence and uniqueness is complete.

\vskip 3mm
\noindent
\textbf{II. Continuity of the data-to-solution map.}
This part of the proof is entirely analogous to that of Subsection \ref{nh-nh-ibvp-wp-h-t-ss} for the case of ``smooth'' data. The proof of Theorem \ref{nh-ibvp-wp-l-t} is complete.

\begin{proof}[\textnormal{\textbf{Proof of  Lemma \ref{nn-complex-l}.}}]
Let $f(z) = |z|^{p-1}z$ with $z=x+i y$ and define  
$$
\varphi(t)=f(w+t(v-w))
=
f_1(w+t(v-w)) +i f_2(w+t(v-w))
 \doteq
\varphi_1(t) + i  \varphi_2(t),
$$
where 
$w=x_1+iy_1$,  $v=x_2+iy_2$ and $0\leqslant t\leqslant 1$.
The mean value theorem for  $\varphi_1$ and $\varphi_2$ implies
$$
|v|^{p - 1}v
-
|w|^{p - 1}w
=
f(v)-f(w)
=
\varphi(1)- \varphi(0)
=
\varphi_1'(\tau_1) +i \varphi_2'(\tau_2)
$$
for some $\tau_1, \tau_2 \in [0, 1]$. In turn, 
$\big| |v|^{p - 1}v  - |w|^{p - 1}w \big|
\leqslant
\left|\varphi'_1(\tau_1)\right|+ \left|\varphi_2'(\tau_2)\right|$
hence, it suffices to estimate $|\varphi_1'(\tau_1)|$ and  $|\varphi_2'(\tau_2)|$.
We compute $\varphi_1'(t)
=
\frac{\partial f_1}{\partial x}(w+t(v-w))\cdot \left(x_2-x_1\right)
+
\frac{\partial f_1}{\partial y}(w+t(v-w))\cdot \left(y_2-y_1\right)$
and
$
\varphi_2'(t)
=
\frac{\partial f_2}{\partial x}(w+t(v-w))\cdot \left(x_2-x_1\right)
+
\frac{\partial f_2}{\partial y}(w+t(v-w))\cdot \left(y_2-y_1\right)$,
where   $\frac{\partial f}{\partial x}
=
|z|^{p-1}
\left[
1+\frac{\left(p-1\right) x^2}{x^2 +y^2} +i  \, \frac{\left(p-1\right) xy}{x^2 +y^2}
\right]
=
\frac{\partial f_1}{\partial x}  +i \frac{\partial f_2}{\partial x}$
and
$
\frac{\partial f}{\partial y}
=
|z|^{p-1}
\left[ 
\frac{\left(p-1\right) xy}{x^2 +y^2}
+
i\, \left(1+\frac{\left(p-1\right) y^2}{x^2 +y^2}\right)
\right]
=
\frac{\partial f_1}{\partial y}  +i \frac{\partial f_2}{\partial y}.$
Thus,  for any $0\leqslant t \leqslant 1$ we have
\ddd{
|\varphi_1'(t)|
&\leqslant
\left|w+t(v-w)\right|^{p-1}
\Big| 
\left[
1+ (p-1)  
\right]  (x_2-x_1)
+
(p-1)  (x_2-x_1)
\Big|
\nn\\
&\leqslant
p \left|(1-t)w+t v\right|^{p-1} |x_2-x_1|  
\leqslant
p \left(|w|+|v|\right)^{p-1} \big( |v-w|\big)
\nn
}
and, noting that 
$\left(|w|+|v|\right)^{p-1} \leqslant 2^p \left(|w|^{p-1}+|v|^{p-1}\right)$ for $p\geqslant 1$,
we deduce
$
|\varphi_1'(t)|
\leqslant
p \cdot 2^p
\big( |w|^{p-1} +|v|^{p-1}\big) |v-w|
$
for all $t\in [0, 1]$.
Similarly,  we have
$|\varphi_2'(t)|
\leqslant
p \cdot 2^p
\big( |w|^{p-1} +|v|^{p-1}\big) |v-w|$ for all $t\in [0, 1]$.
These two estimates combined with the mean value inequality above imply  the inequality of Lemma \ref{nn-complex-l}.
\end{proof}

%
%
%
%
%%%%%%%%%%%%%%%%%%%%%%%%  
%
%	Linear IBVP Estimates on the Interval
%
%%%%%%%%%%%%%%%%%%%%%%%%  
%
%
%
\section{Linear IBVP Estimates on the Finite Interval}
\label{nh-fi-ibvp-le-s}

The analysis of the forced linear heat equation on the finite interval can be simplified significantly by exploiting the results of Theorem \ref{nh-flh-ibvp-wp-t} for this equation on the half-line. 
In particular, let $U_0\in H_x^s(\mathbb R)$ be a whole line extension of the initial data $u_0\in H_x^s(0, \ell)$ of the  interval IBVP \eqref{nh-fi-flh-ibvp} such that $\no{U_0}_{H_x^s(\mathbb R)} \leqslant c \no{u_0}_{H_x^s(0, \ell)}$. 
For $s>\frac 12$, use a similar argument  to extend the forcing $f$ of  \eqref{nh-fi-flh-ibvp} by $F\in C([0, T]; H_x^s(\mathbb R))$ such that $\no{F}_{C([0, T];H_x^s(\mathbb R))} \leqslant c \no{f}_{C([0, T]; H_x^s(0, \ell))}$, while for $s\leqslant \frac 12$ simply extend $f$ by zero. Subsequently, restrict $U_0$ and $F$ to the half-line to obtain  initial data and forcing for the half-line IBVP \eqref{nh-flh-ibvp}.
Then, the solution $S\big[u_0, g_0, h_0; f]$ of the interval IBVP \eqref{nh-fi-flh-ibvp} can be expressed as 
\eee{\label{nh-fi-flh-ibvp-sup}
S\big[u_0, g_0, h_0; f\big] 
= 
S\big[U_0\big|_{x\in [0, \ell]}, g_0; F\big|_{x\in [0, \ell]}\big]\Big|_{x\in[0, \ell]} 
+ 
S\big[0, 0, w_0; 0\big],
}
where $S\big[U_0\big|_{x\in [0, \ell]}, g_0; F\big|_{x\in [0, \ell]}]$ is the solution of the half-line IBVP \eqref{nh-flh-ibvp} restricted on $[0, \ell]$ and $S\big[0, 0, w_0; 0\big]$ satisfies the reduced interval IBVP
\sss{\label{nh-fi-ibvp-r}
\ddd{
& u_t - u_{xx} = 0,   && x \in (0, \ell), \ t \in (0, T), 
\label{nh-fi-flh-ibvp-r-eq-nh}
\\
&u(x,0) = 0,  && x \in [0, \ell],
\label{nh-fi-flh-ibvp-r-ic-nh}
\\ 
&u(0,t) = 0, && t \in [0, T],
\\ 
&u(\ell,t) =  h_0(t) - S\big[U_0\big|_{x\in [0, \ell]}, g_0; F\big|_{x\in [0, \ell]}](\ell, t)  \doteq w_0(t), \quad && t \in [0, T].
\label{nh-fi-flh-ibvp-r-bc-nh}
}
}
Since the half-line solution $S\big[U_0\big|_{x\in [0, \ell]}, g_0; F\big|_{x\in [0, \ell]}\big]$ was estimated in the preceding sections, we only need to estimate the reduced IBVP \eqref{nh-fi-ibvp-r}.

Note that for all $0\leqslant s < \tfrac 32$, $s\neq \tfrac 12$, the boundary datum $w_0$ of IBVP \eqref{nh-fi-ibvp-r} belongs to $H_t^{\frac{2s+1}{4}}(0, T)$ due to the fact that both $h_0$ and $S\big[U_0\big|_{x\in [0, \ell]}, g_0; F\big|_{x\in [0, \ell]}]$ belong to $H_t^{\frac{2s+1}{4}}(0, T)$. In particular, combining the extension  inequalities stated above and the half-line  estimates \eqref{nh-flh-ibvp-te}, we have 
\sss{\label{nh-fi-w0}
\ddd{
\no{w_0}_{H_t^{\frac{2s+1}{4}}(0, T)}
&\leqslant
c_{s, p}
\Big(
\no{u_0}_{H_x^s(0, \ell)}
+
\no{g_0}_{H^{\frac{2s+1}{4}}_t(0, T)}
+
\no{h_0}_{H^{\frac{2s+1}{4}}_t(0, T)}
\nn\\
&\hskip 1.3cm
+
\sqrt T \prod_{j=1}^p  \no{f_j}_{C([0, T]; H_x^s(0, \ell))}
\Big),
\quad
\tfrac 12<s<\tfrac 32, 
\label{nh-fi-w0-h}
}
and
\ddd{
\no{w_0}_{H_t^{\frac{2s+1}{4}}(0, T)}
&\leqslant
c_{s, p}
\Big(
\no{u_0}_{H_x^s(0, \ell)}
+
\no{g_0}_{H^{\frac{2s+1}{4}}_t(0, T)}
+
\no{h_0}_{H^{\frac{2s+1}{4}}_t(0, T)}
\nn\\
&\hskip 1.5cm
+
\sqrt T \prod_{j=1}^p  \no{f_j}_{C^\alpha([0, T]; L_x^{p}(0, \ell))}\Big),
\quad
0\leqslant s<\tfrac 12.
\label{nh-fi-w0-l}
}
}

The analysis can be further simplified by ``embedding'' IBVP \eqref{nh-fi-ibvp-r} inside the \textit{pure IBVP}
\sss{\label{nh-fi-ibvp-rv}
\ddd{
&v_t-v_{xx} = 0, &&x\in (0,\ell), \ t\in (0,2),
\label{nh-fi-ibvp-rv-eq} \\
&v(x,0)= 0, &&x\in [0,\ell],  \label{nh-fi-ibvp-rv-ic} \\
&v(0,t) = 0, && t\in [0, 2],
\\[-.5em]
&v(\ell,t) = h(t)\in H_t^{\frac{2s+1}{4}}(\mathbb R),
\quad && 
t\in [0, 2],
\label{nh-fi-ibvp-rv-bc}
}
}
where $h$ is an extension of $w_0$ constructed analogously to the extension $g$ of Section \ref{nh-ibvp-le-s}  so that $\text{supp}(h)\subset (0, 2)$ and 
\eee{\label{nh-fi-bcdet}
\no{h}_{H_t^{\frac{2s+1}{4}}(\mathbb R)}
\leqslant 
c_s 
\no{w_0}_{H_t^{\frac{2s+1}{4}}(0, T)},
\quad
0\leqslant s < \tfrac 32, \ s\neq\tfrac 12.
}
Indeed, IBVP \eqref{nh-fi-ibvp-rv} restricted on $(0, \ell) \times (0, T)$ becomes IBVP \eqref{nh-fi-ibvp-r}.
Therefore,  IBVP \eqref{nh-fi-ibvp-r} can be estimated via IBVP \eqref{nh-fi-ibvp-rv}, whose solution is given by the UTM formula \eqref{nh-fi-flh-utm-sol-T} as  
\ddd{\label{nh-fi-ibvp-rv-sol}
v(x, t) 
&=
S\big[0, 0, h; 0\big](x, t) 
\nn\\
&=
\frac{1}{\pi}\int_{k=0}^\infty 
\frac{e^{ia^3kx+ik^2 t}}{e^{ia^3k\ell}-e^{-ia^3k\ell}}
\, k   \widehat h(k^2) dk
+ \frac{1}{\pi}\int_{k=0}^\infty 
\frac{e^{iakx-ik^2 t}}{e^{iak\ell}-e^{-iak\ell}}
\, k \widehat h(-k^2)  dk
\nn\\
&\quad
+\frac{1}{\pi}\int_{k=0}^\infty 
\frac{e^{-ia^3k x+ik^2 t}}{e^{-ia^3k\ell}-e^{ia^3k\ell}}
\, k   \widehat h(k^2)    dk
+ \frac{1}{\pi}\int_{k=0}^\infty 
\frac{e^{-iakx-ik^2 t}}{e^{-iak\ell}-e^{iak\ell}}
\, k   \widehat h(-k^2)   dk
}
with $a= e^{i\frac{\pi}{4}}$ and the contours $\p D^\pm$ depicted in Figure \ref{nh-fi-flh-dpm}. Observe that, similarly to the half-line,  we have exploited the compact support of $h$ in order to replace in formula \eqref{nh-fi-ibvp-rv-sol} the time transform  $\widetilde h_0$ by the  Fourier transform of $h$ on the whole line.

Starting from   formula \eqref{nh-fi-ibvp-rv-sol}, we shall establish the following Sobolev estimates for IBVP \eqref{nh-fi-ibvp-rv}. 

\begin{theorem}[\b{Sobolev-type estimates for the pure IBVP}]
\label{nh-fi-ibvp-rv-t}
The solution $v=S\big[0, 0, h;0\big]$ of the pure IBVP \eqref{nh-fi-ibvp-rv} given by the UTM formula \eqref{nh-fi-ibvp-rv-sol} satisfies the space and time estimates
\ddd{
\no{v}_{C([0, 2]; H_x^s(0,\ell))}
\leqslant
c_s \no{h}_{H_t^{\frac{2s+1}{4}}(\mathbb R)},
\quad && s\geqslant 0,
\label{nh-fi-ibvp-rv-se}
\\
\no{v}_{C([0,\ell]; H_t^{\frac{2s+1}{4}}(0,2))} 
\leqslant
c_s \no{h}_{H_t^{\frac{2s+1}{4}}(\mathbb R)},
\quad && s \in \mathbb R.
\label{nh-fi-ibvp-rv-te}
}
\end{theorem}

\begin{proof}[\textnormal{\textbf{Proof of Theorem \ref{nh-fi-ibvp-rv-t}.}}]
The main ideas behind the proof are the same with those that form the backbone of the proof in the case of the half-line  (see Theorem \ref{nh-ibvp-rv-t}). However, extra care is now needed when handling the exponential differences in the denominators of the solution formula \eqref{nh-fi-ibvp-rv-sol}.

We begin with the space estimate \eqref{nh-fi-ibvp-rv-se}. Grouping together the first with the third and the second with the fourth term of \eqref{nh-fi-ibvp-rv-sol}, we write $v=v_1+v_2$ where
\sss{
\ddd{
\nn
v_1(x, t)
&=
\int_{k=0}^\infty 
\frac{e^{i\gamma_1kx}-e^{-i\gamma_1k x}}{e^{i\gamma_1k\ell}-e^{-i\gamma_1k\ell}}
\, G_1(k, t) dk,
&&
\gamma_1 = a^3, \
G_1(k, t) 
= 
\frac{1}{\pi} e^{ik^2 t} k   \widehat h(k^2),
\\
v_2(x, t)
&=
\int_{k=0}^\infty 
\frac{e^{i\gamma_2 kx}-e^{-i\gamma_2kx}}{e^{i\gamma_2k\ell}-e^{-i\gamma_2k\ell}}
\, G_2(k, t)  dk,
\quad
&&\gamma_2 = a, \
G_2(k, t) 
= 
\frac{1}{\pi} e^{-ik^2 t} k   \widehat h(-k^2).
\nn
}
}
As the estimation of $v_1$ and $v_2$ is analogous, we only provide the details for $v_1$.
We employ the physical space definition of the $H_x^s(0,\ell)$-norm:
\eee{
\label{nh-fi-frac-sob-def}
\no{v_1(t)}_{H_x^s(0, \ell)}
=
\sum_{j=0}^{\left\lfloor s \right\rfloor}
\no{\p_x^j v_1(t)}_{L_x^2(0, \ell)}
+
\big\| \p_x^{\left\lfloor s \right\rfloor}v_1(t)\big\|_\beta,\quad s=\left\lfloor s \right\rfloor +\beta\geqslant 0, \ 0\leqslant \beta < 1, 
}
where $\lfloor \cdot \rfloor$ denotes the floor function
and the fractional norm $\no{\cdot}_\beta$ is defined by 
\eee{
\no{v_1(t)}_{\beta}
=
\left(\int_{x=0}^\ell \int_{z=0}^{\ell-x} \dfrac{\left| v_1(x+z, t)- v_1(x, t)\right|^2}{z^{1+2\beta}}\, dzdx\right)^{\frac 12}, \quad 0 < \beta < 1.
\nn
}
\noindent
\textit{\b{(i) The case $\left\lfloor s \right\rfloor = 0$.}} 
Since $s=\beta\in [0, 1)$, we only need to estimate $\no{v_1(t)}_\beta$ and $\no{v_1(t)}_{L_x^2(0, \ell)}$.
For the fractional norm $\no{v_1(t)}_\beta$, we write
\eee{\label{nh-fi-j1-j2-split}
v_1(x, t)
=
J_1(x, t) + J_2(x, t)
}
with
$$
J_1(x, t)
=
\int_{k=0}^1
\frac{e^{i\gamma_1kx}-e^{-i\gamma_1k x}}{e^{i\gamma_1k\ell}-e^{-i\gamma_1k\ell}}
\, G_1(k, t) dk,
\quad
J_2(x, t)
=
\int_{k=1}^\infty 
\frac{e^{i\gamma_1kx}-e^{-i\gamma_1k x}}{e^{i\gamma_1k\ell}-e^{-i\gamma_1k\ell}}
\, G_1(k, t)  dk.
$$

For $J_1$, we have
\ddd{
\no{J_1(t)}_\beta
&=
\left( \int_{z=0}^\ell \int_{x=0}^{\ell-z}
\frac{1}{z^{1+2\beta}}
\left|
\int_{k=0}^1 
\frac{\left(e^{i\gamma_1k z}-1\right)
\left[
e^{i\gamma_1k (x+\ell)}
+
e^{-i\gamma_1k (x+z-\ell)} 
\right]}{e^{2i\gamma_1k\ell}-1}
\, G_1(k,t) dk
 \right|^2
 dxdz \right)^{\frac 12}
\nonumber\\
&\leqslant
\int_{k=0}^1
\left(
\int_{z=0}^\ell \int_{x=0}^{\ell-z}  \frac{1}{z^{1+2\beta}}\,
\frac{\left|1-e^{i\gamma_1k z}\right|^2}{\big|1-e^{2i\gamma_1k\ell}\big|^2}
 \left|G_1(k ,t)\right|^2 dx dz \right)^{\frac 12}
dk 
\nn
}
after noting that $x+\ell\geqslant 0$ and $x+z-\ell = x-(\ell-z) \leqslant 0$ and applying Minkowski's integral inequality in $k$.
Hence, using Lemma \ref{nh-gamma-l}, integrating in $x$ and employing estimate \eqref{nh-gamma-est}, we obtain
\eee{
\no{J_1(t)}_\beta
\lesssim
\sqrt \ell 
\int_{k=0}^1
\frac{\left|G_1(k ,t)\right|}{1-e^{-\sqrt 2 k\ell}} 
\left(
\int_{z=0}^\ell 
\frac{\big(1-e^{-\frac{\sqrt 2}{2} k z}\big)^2}{z^{1+2\beta}}  \, dz \right)^{\frac 12}
dk
\lesssim 
\int_{k=0}^1
\frac{k^{\frac 12+\beta}}{1-e^{-\sqrt 2 k\ell}} \, k^{\frac 12}  
\big|\widehat h(k^2)\big|
dk.
\nn
}
Then, applying the Cauchy-Schwarz inequality and noting that
$
\int_{k=0}^1
\frac{k^{1+2\beta}}{\left(1-e^{-\sqrt 2 k\ell}\right)^2} \, dk
<\infty
$
since the singularity at $k=0$ is removable and the domain of integration is compact, we deduce
\eee{\label{nh-fi-j1-est}
\no{J_1(t)}_\beta
\lesssim
\no{h}_{L_t^2(\mathbb R)}
\leqslant
\no{h}_{H_t^{\frac{2\beta+1}{4}}(\mathbb R)},
\quad 0 < \beta < 1.
}

Regarding $J_2$, writing
$
J_2 
=
J_{21} + J_{22} 
$
with
\sss{
\ddd{
J_{21}(x, t) 
&=
\int_{k=1}^\infty 
e^{ia^3kx} K_1(k, t)  dk,
\quad
&&K_1(k, t) = \frac{G_1(k, t)}{e^{ia^3k\ell}-e^{-ia^3k\ell}},
\nn\\
J_{22}(x, t) 
&=
\int_{k=1}^\infty 
e^{-ia^3k(x-\ell)}  K_2(k, t) dk,
\quad
&&K_2(k, t) = \frac{G_1(k, t)}{1-e^{2ia^3k\ell}}
\nn
}
}
and employing once again Lemma \ref{nh-gamma-l}, we infer
\eee{
\no{J_{21}(t)}_\beta^2
\lesssim
\int_{z=0}^\ell
\frac{1}{z^{1+2\beta}} 
\no{
\int_{k=1}^\infty 
e^{-\frac{\sqrt 2}{2}kx}
\big(1-e^{-\frac{\sqrt 2}{2}kz} \big) \left|K_1(k, t)\right|  dk
}_{L_x^2(0, \infty)}^2 dz.
\nn
}
Hence, combining Lemma \ref{nh-tlap-l} for the $L^2$-boundedness of the Laplace transform with estimate \eqref{nh-gamma-est}, we find
\eee{\label{fi-v1-bound-L3}
\no{J_{21}(t)}_\beta^2
\lesssim
\int_{k=1}^\infty \frac{k^2 \big|\widehat h(k^2)\big|^2 k^{2\beta}}{\left|e^{ia^3k\ell}-e^{-ia^3k\ell}\right|^2} dk
\lesssim
\int_{k=1}^\infty \frac{k^2 \big|\widehat h(k^2)\big|^2 k^{2\beta}}{\big(e^{\frac{\sqrt 2}{2} k\ell}-e^{-\frac{\sqrt 2}{2} k\ell}\big)^2} \, dk,
}
with the last inequality due to the fact that
$
\big|e^{ia^3k\ell}-e^{-ia^3k\ell}\big|
\geqslant
e^{\frac{\sqrt 2}{2} k\ell}-e^{-\frac{\sqrt 2}{2} k\ell}$
for $k\geqslant 0$.

Next, let 
$
\psi(k)
=
e^{\frac{\sqrt 2}{2} k\ell}-e^{-\frac{\sqrt 2}{2} k\ell}
$
and note that $\psi(k)>0$ for $k\geqslant 1$ (importantly, $\psi(k)=0$ only for $k=0$, i.e. the integrand of \eqref{fi-v1-bound-L3} is non-singular). Moreover, observe that $\psi$ is infinitely differentiable and, in particular,   
$\psi$ is increasing on $[1, \infty)$. Hence, 
$
0 <  e^{\frac{\sqrt 2}{2} \ell}-e^{-\frac{\sqrt 2}{2} \ell} 
\leqslant
e^{\frac{\sqrt 2}{2} k\ell}-e^{-\frac{\sqrt 2}{2} k\ell}$ for $k\geqslant 1$.
Thus, back to \eqref{fi-v1-bound-L3}, we have
\eee{\label{nh-fi-j21-est}
\no{J_{21}(t)}_\beta^2
\lesssim
\int_{k=1}^\infty \frac{k^2 \big|\widehat h(k^2)\big|^2 k^{2\beta}}{\left(e^{\frac{\sqrt 2}{2} \ell}-e^{-\frac{\sqrt 2}{2} \ell}\right)^2} \, dk
\simeq
\int_{k=1}^\infty k^{1+2\beta} \big|\widehat h(k^2)\big|^2 k dk
\lesssim
\no{h}_{H_{t}^{\frac{2\beta+1}{4}}(\mathbb R)}^2.
}

Concerning $J_{22}$, we note that $\no{J_{22}(t)}_\beta = \big\|\widetilde J_{22}(t)\big\|_{\beta}$
where
$
\widetilde J_{22}(x, t) 
\doteq
J_{22}(\ell-x, t).
%=
%\frac{1}{\pi}\int_{k=1}^\infty e^{ia^3k x}  K_2(k, t) dk.
$
Therefore, 
\eee{
\no{J_{22}(t)}_\beta^2
=
\big\|\widetilde J_{22}(t)\big\|_{\beta}^2
\nn\\
\lesssim
\int_{x=0}^\ell\int_{z=0}^{\ell-x} \frac{1}{z^{1+2\beta}} 
\left(
\int_{k=1}^\infty 
e^{-\frac{\sqrt 2}{2}k x}
\big|e^{ia^3kz}-1\big| \left|K_2(k, t)\right|  dk
\right)^2
dzdx
\nn
}
and employing once again Lemma \ref{nh-gamma-l} we infer
\eee{ 
\no{J_{22}(t)}_\beta^2
\lesssim
\int_{z=0}^\ell
\frac{1}{z^{1+2\beta}} 
\no{
\int_{k=1}^\infty 
e^{-\frac{\sqrt 2}{2}k x}
\big(1-e^{-\frac{\sqrt 2}{2}kz} \big) \left|K_2(k, t)\right|  dk}_{L_x^2(0, \infty)}^2 dz.
\nn
}
The Laplace transform Lemma \ref{nh-tlap-l}, estimate \eqref{nh-gamma-est}  and the fact that $\big|1-e^{2ia^3k\ell}\big| \geqslant 1 - e^{-\sqrt 2 k\ell}$ for all $k\geqslant 0$ imply
\eee{\nn
\no{J_{22}(t)}_\beta^2
\lesssim
\int_{z=0}^{\frac{\sqrt 2}{2} \ell}  \frac{1}{z^{1+2\beta}} \no{Q_{z,t}}^2_{L^2_{k}(0,\infty)}dz
\lesssim
\int_{k=1}^\infty 
\frac{k^2 \big|\widehat h(k^2)\big|^2}{\big(1 - e^{-\sqrt 2 k\ell}\big)^2}
\, k^{2\beta} dk.
}
Similarly to the argument used earlier for $J_{21}$, for $k\geqslant 1$ we have
$0 <  1-e^{- \sqrt 2  \ell}
\leqslant
1-e^{- \sqrt 2  k\ell}$.
Hence, 
\eee{\label{nh-fi-j22-est}
\no{J_{22}(t)}_\beta^2
\lesssim
\int_{k=1}^\infty 
\frac{k^2 \big|\widehat h(k^2)\big|^2}{\big(1 - e^{-\sqrt 2\ell}\big)^2}
\, k^{2\beta} dk
\leqslant
\int_{k=1}^\infty k^{1+2\beta} \big|\widehat h(k^2)\big|^2 k dk\lesssim
\no{h}_{H_{t}^{\frac{2\beta+1}{4}}(\mathbb R)}^2.
}

Overall, combining estimates \eqref{nh-fi-j1-est}, \eqref{nh-fi-j21-est} and \eqref{nh-fi-j22-est} with the writing \eqref{nh-fi-j1-j2-split}, we find
\eee{\label{nh-fi-v1-frac}
\no{v_1(t)}_{\beta}
\lesssim
\no{h}_{H_t^{\frac{2s+1}{4}}(\mathbb R)},
\quad s=\beta\in (0, 1).
}

The norm $\no{v_1(t)}_{L_x^2(0, \ell)}$ will also be estimated using the splitting  \eqref{nh-fi-j1-j2-split} .
In particular, for $J_1$ we employ Lemma \ref{nh-gamma-l} to infer
\eee{\label{nh-fi-j1-l2-0}
\no{J_1(t)}_{L_x^2(0,\ell)}^2
\lesssim
\int_{x=0}^\ell
\left(
\int_{k=0}^1 
\frac{e^{\frac{\sqrt 2}{2} kx}-e^{-\frac{\sqrt 2}{2}k x}}{\left|e^{ia^3k\ell}-e^{-ia^3k\ell}\right|}
\, k \big|\widehat h(k^2)\big|  dk
\right)^2
dx.
}
The ratio of exponentials involved in the $k$-integral is bounded by $1$ for all $k\geqslant 0$. Thus, applying also Cauchy-Schwarz in $k$, we obtain
\eee{
\no{J_1(t)}_{L_x^2(0,\ell)}^2
\lesssim
\int_{x=0}^\ell
\left(
\int_{\tau=0}^1
\big|\widehat h(\tau)\big|^2  d\tau
\right)
dx
\lesssim
\no{h}_{H_t^{\frac{2s+1}{4}}(\mathbb R)}^2,
\quad s\geqslant -\tfrac 12.
\label{nh-fi-j1-l2}
}
Concerning $J_2$, similarly to \eqref{nh-fi-j1-l2-0} we have
$$
\no{J_2(t)}_{L_x^2(0,\ell)}^2
\lesssim
\int_{x=0}^\ell
\left(
\frac{1}{\pi}\int_{k=1}^\infty 
\frac{e^{\frac{\sqrt 2}{2} kx}-e^{-\frac{\sqrt 2}{2}k x}}{\left|e^{ia^3k\ell}-e^{-ia^3k\ell}\right|}
\, k \big|\widehat h(k^2)\big|  dk
\right)^2
dx.
$$
Note that
$
\big|e^{ia^3k\ell}-e^{-ia^3k\ell}\big|
\geqslant
e^{\frac{\sqrt 2}{2} k\ell}-e^{-\frac{\sqrt 2}{2} k\ell}
$
and, furthermore,
\eee{\label{nh-fi-exp-bound-2}
\frac{1}{e^{\frac{\sqrt 2}{2} k\ell}-e^{-\frac{\sqrt 2}{2} k\ell}}
\leqslant
\frac{e^{-\frac{\sqrt 2}{2} k\ell}}{1-e^{-\sqrt 2 \ell}}, \quad k\geqslant 1.
}
Using this bound and the fact that $e^{-\frac{\sqrt 2}{2} k \ell} < 1$ for all $k\geqslant 1$, we find 
\ddd{ 
\no{J_2(t)}_{L_x^2(0,\ell)}^2
&\lesssim
\int_{x=0}^\ell
\left(
\int_{k=1}^\infty 
\left(e^{\frac{\sqrt 2}{2} kx}-e^{-\frac{\sqrt 2}{2}k x}\right)
e^{-\frac{\sqrt 2}{2}k \ell}
\, k \big|\widehat h(k^2)\big|  dk
\right)^2
dx
%\nn\\
%&\lesssim
%\int_{x=0}^\ell
%\left(
%\int_{k=1}^\infty e^{\frac{\sqrt 2}{2} k(x-\ell)} \, k \big|\widehat h(k^2)\big|  dk
%\right)^2
%dx
%+
%\int_{x=0}^\ell
%\left(
%\int_{k=1}^\infty e^{-\frac{\sqrt 2}{2} k(x+\ell)} \, k \big|\widehat h(k^2)\big|  dk
%\right)^2
%dx
\nn\\
&\lesssim
\int_{x=0}^{\frac{\sqrt 2}{2} \ell}
\left(
\int_{k=1}^\infty e^{-k x} \, k \big|\widehat h(k^2)\big|  dk
\right)^2
dx
\leqslant
\int_{x=0}^\infty
\left(
\int_{k=1}^\infty e^{-k x} \, k \big|\widehat h(k^2)\big|  dk
\right)^2
dx.
\nn
}
Therefore,  by the Laplace  transform Lemma \ref{nh-tlap-l} we get
\eee{
\no{J_2(t)}_{L_x^2(0,\ell)}^2
\lesssim
\int_{k=1}^\infty k^2 \big|\widehat h(k^2)\big|^2  dk
=
\int_{\tau=1}^\infty \tau^{\frac 12} \big|\widehat h(\tau)\big|^2  d\tau
\leqslant
\no{h}_{H_t^{\frac{2s+1}{4}}(\mathbb R)}^2, \quad s\geqslant 0.
\label{nh-fi-j2-l2}
}
Combining \eqref{nh-fi-j1-l2} and \eqref{nh-fi-j2-l2}, we obtain
\eee{\label{nh-fi-v1-l2}
\no{v_1(t)}_{L_x^2(0,\ell)}
\lesssim
\no{h}_{H_t^{\frac{2s+1}{4}}(\mathbb R)}, \quad s\geqslant 0.
}
%
%\vskip 3mm
\noindent
\textit{\b{(ii) The case $\left\lfloor s \right\rfloor>0$.}}
Now $s=\left\lfloor s \right\rfloor+\beta$ with $\left\lfloor s \right\rfloor\in \mathbb N\setminus\{0\}$ and $\beta\in[0, 1)$. Thus, according to definition \eqref{nh-fi-frac-sob-def} we need to estimate the fractional norm $\big\|\p_x^{\left\lfloor s \right\rfloor} v_1(t)\big\|_{\beta}$ and also the $L^2$-norm $\big\|\p_x^j v_1(t)\big\|_{L_x^2(0, \ell)}$  for all integers $0\leqslant j \leqslant \left\lfloor s \right\rfloor$ . Both of those norms can be handled in exactly the same way as the norms $\no{v_1(t)}_\beta$ and $\no{v_1(t)}_{L_x^2(0, \ell)}$ that were estimated in case (i) above, eventually yielding the space estimate \eqref{nh-fi-ibvp-rv-se} for $v_1$ for all $s\geqslant 0$. As noted earlier, the estimation of $v_2$ is similar.

\vskip 3mm

Concerning the time estimate \eqref{nh-fi-ibvp-rv-te}, making 
the change of variable $k=\sqrt{\tau}$ in  formula \eqref{nh-fi-ibvp-rv-sol} we infer that  the temporal Fourier transform of $v$ is given by
\eee{
\widehat v(x, \tau)
\simeq
\left\{
\def\arraystretch{2.2}
\begin{array}{ll}
\dfrac{e^{ia^3\sqrt \tau x}-e^{-ia^3\sqrt \tau x}}{e^{ia^3\sqrt \tau\ell}-e^{-ia^3\sqrt \tau\ell}}\, \widehat h(\tau), & \tau\geqslant 0, 
\\
\dfrac{e^{ia\sqrt{-\tau} x}-e^{-ia\sqrt{-\tau} x}}{e^{ia\sqrt{-\tau}\ell}-e^{-ia\sqrt{-\tau}\ell}}\,   \widehat h(\tau), & \tau < 0.
\end{array}
\right.
\nn
}
Hence, by the definition of the $H_t^{\frac{2s+1}{4}}(\mathbb R)$-norm it follows that
\ddd{
\no{v(x)}_{H_t^{\frac{2s+1}{4}}(\mathbb R)}^2
&\lesssim
\int_{\tau=0}^\infty 
\left(1+\tau^2\right)^{\frac{2s+1}{4}}
\left|
\frac{e^{ia^3\sqrt \tau x}-e^{-ia^3\sqrt \tau x}}{e^{ia^3\sqrt \tau\ell}-e^{-ia^3\sqrt \tau\ell}}\, \widehat h(\tau)
\right|^2 d\tau
\nn\\
&\quad
+
\int_{\tau=-\infty}^0 
\left(1+\tau^2\right)^{\frac{2s+1}{4}}
\left|
\frac{e^{ia\sqrt{-\tau} x}-e^{-ia\sqrt{-\tau} x}}{e^{ia\sqrt{-\tau}\ell}-e^{-ia\sqrt{-\tau}\ell}}\,   \widehat h(\tau)
\right|^2 d\tau.
\nn
}
Lemma \ref{nh-gamma-l} and the fact that $k\geqslant 0$ imply that the ratios of exponentials in the above integrals are bounded by $1$. Therefore, 
$$
\no{v(x)}_{H_t^{\frac{2s+1}{4}}(\mathbb R)}^2
\lesssim
\int_{\tau=0}^\infty 
\left(1+\tau^2\right)^{\frac{2s+1}{4}}
\big|\widehat h(\tau)\big|^2 d\tau
+
\int_{\tau=-\infty}^0 
\left(1+\tau^2\right)^{\frac{2s+1}{4}}
\big|\widehat h(\tau)\big|^2 d\tau
=
\no{h}_{H_t^{\frac{2s+1}{4}}(\mathbb R)}^2,
$$
which implies the time estimate \eqref{nh-fi-ibvp-rv-te}   since $H_t^{\frac{2s+1}{4}}(0, T)$ is the restriction of $H_t^{\frac{2s+1}{4}}(\mathbb R)$  on $(0, T)$.

As for the half-line, continuity in $t$ and $x$ follows along the lines of Theorem 5 of \cite{fhm2017}.
\end{proof}

We complete the analysis of IBVP \eqref{nh-fi-ibvp-rv} with the estimation in  $C^\alpha([0, T]; L_x^{p}(0, \ell))$.
\begin{theorem}[\b{$C^\alpha([0, T]; L_x^{p}(0, \ell))$-estimate for the pure IBVP}]
\label{nh-fi-ibvp-rv-ca-t}
If $\frac 12-\frac{1}{p} < s < \frac 12$, then the solution $v = S\big[0, 0, h; 0\big]$ of the pure IBVP \eqref{nh-fi-ibvp-rv} given by the UTM formula \eqref{nh-fi-ibvp-rv-sol} admits the estimate
\eee{\label{nh-fi-ibvp-rv-ca-est-0}
\no{v}_{C^\alpha([0, T]; L_x^{p}(0, \ell))}
\lesssim 
T^\alpha  \no{h}_{H_t^{\frac{2s+1}{4}}(\mathbb R)}.
}
\end{theorem}

\begin{proof}[\textnormal{\textbf{Proof of Theorem \ref{nh-fi-ibvp-rv-ca-t}.}}]
We  write $v=v_1+v_2$ as in the proof of Theorem \ref{nh-fi-ibvp-rv-t} and note that $v_1$ and $v_2$ can be estimated in the same way; hence, we only provide the proof for $v_1$.  By Minkowski's integral, we have
\sss{
\ddd{
\no{v_1(t)}_{L_x^{p}(0, \ell)}
&\lesssim
\int_{k=0}^1 
k \big|\widehat h(k^2)\big|
\left(\int_{x=0}^\ell  \left|
\frac{e^{ia^3kx}-e^{-ia^3k x}}{e^{ia^3k\ell}-e^{-ia^3k\ell}}
\right|^{p}  dx \right)^{\frac{1}{p}} 
dk
\label{nh-fi-ca-1}
\\
&\quad
+\int_{k=1}^\infty 
k \big|\widehat h(k^2)\big|
\left(\int_{x=0}^\ell  \left|
\frac{e^{ia^3kx}-e^{-ia^3k x}}{e^{ia^3k\ell}-e^{-ia^3k\ell}}
\right|^{p}  dx \right)^{\frac{1}{p}} 
dk.
\label{nh-fi-ca-2}
}
}
For \eqref{nh-fi-ca-1}, we bound the ratio of exponentials involved in the $x$-integral  by $1$ and then apply Cauchy-Schwarz in $k$ to obtain
\eee{\label{nh-fi-ca-1-est}
\eqref{nh-fi-ca-1}
\lesssim
\ell^{\frac{1}{p}}
\int_{k=0}^1 
k \big|\widehat h(k^2)\big| dk
\lesssim
\no{h}_{H_t^{\frac{2s+1}{4}}(\mathbb R)}^2,
\quad s\geqslant -\tfrac 12.
}
For \eqref{nh-fi-ca-2}, we employ the bound \eqref{nh-fi-exp-bound-2} and note that  $e^{-\sqrt 2 k x}\leqslant 1$ for $k, x\geqslant 0$ to infer
\eee{
\eqref{nh-fi-ca-2}
\lesssim
\int_{k=1}^\infty 
k \big|\widehat h(k^2)\big|
\left(\int_{x=0}^\ell  
e^{\frac{\sqrt 2}{2} k(x-\ell)p} dx \right)^{\frac{1}{p}} 
dk
\lesssim
\int_{k=1}^\infty 
k^{1-\frac{1}{p}} \big|\widehat h(k^2)\big|  dk.
\nn
}
Then, letting $k=\sqrt \tau$ and using Cauchy-Schwarz, we get
\ddd{
\eqref{nh-fi-ca-2}
%&\lesssim
%\int_{\tau=1}^\infty 
%\tau^{-\frac{1}{2p}} \big|\widehat h (\tau)\big|  d\tau
%\nn\\
&\lesssim
%\left(\int_{\tau=1}^\infty 
%\tau^{-\frac{1}{p}} \left(1+\tau^2\right)^{-\frac{2s+1}{4}} d\tau\right)^{\frac 12}
%\left(\int_{\tau=1}^\infty 
%\left(1+\tau^2\right)^{\frac{2s+1}{4}} \big|\widehat  h(\tau)\big|^2  d\tau
%\right)^{\frac 12}
%\nn\\
%&\leqslant
\left(
\int_{\tau=1}^\infty 
\tau^{-\frac{1}{p}} \left(1+\tau^2\right)^{-\frac{2s+1}{4}} d\tau\right)^{\frac 12} 
\no{h}_{H_t^{\frac{2s+1}{4}}(\mathbb R)}
\lesssim
\no{h}_{H_t^{\frac{2s+1}{4}}(\mathbb R)},\label{nh-fi-ca-2-est}
}
provided that $s>\tfrac 12-\tfrac{1}{p}$.
Combining estimates \eqref{nh-fi-ca-1-est} and \eqref{nh-fi-ca-2-est} for this range of $s$, we obtain
$
\no{v_1(t)}_{L_x^{p}(0, \ell)}
\lesssim
\no{h}_{H_t^{\frac{2s+1}{4}}(\mathbb R)}$,
which implies estimate \eqref{nh-fi-ibvp-rv-ca-est-0} for $v_1$ since $\alpha>0$.
Note that the restriction $s<\frac 12$ comes from the condition on $b$  in the definition \eqref{nh-cbeta-def} of $\alpha$ and not from the present proof.
\end{proof}

Overall,  in view of Theorems \ref{nh-fi-ibvp-rv-t} and \ref{nh-fi-ibvp-rv-ca-t},  inequalities \eqref{nh-fi-w0} and \eqref{nh-fi-bcdet}, and Theorem \ref{nh-flh-ibvp-wp-t} for the half-line IBVP \eqref{nh-flh-ibvp}, the superposition \eqref{nh-fi-flh-ibvp-sup}  yields the following result for the forced linear IBVP \eqref{nh-fi-flh-ibvp} on the finite interval.

\begin{theorem}[\b{Estimates for the forced linear heat on the finite interval}] 
\label{nh-fi-flh-ibvp-wp-t}
For $f=\prod_{j=1}^p  f_j$, the solution $u=S\big[u_0, g_0, h_0; f\big]$ of the forced linear heat IBVP  \eqref{nh-fi-flh-ibvp}  admits the space estimates
\sss{\label{nh-fi-flh-ibvp-se}
\ddd{
\no{S\big[u_0, g_0, h_0; f\big]}_{C([0, T]; H_x^s(0, \ell))}
&\leqslant
c_{s, p}
\Big(
\no{u_0}_{H_x^s(0, \ell)}
+
\no{g_0}_{H^{\frac{2s+1}{4}}_t(0, T)}
+
\no{h_0}_{H^{\frac{2s+1}{4}}_t(0, T)}
\label{nh-fi-flh-ibvp-se-h}
\\
&\hskip 1.3cm
+
\sqrt T \prod_{j=1}^p  \no{f_j}_{C([0, T]; H_x^s(0, \ell))}\Big),
\quad
\tfrac 12<s<\tfrac 32,
\nn\\
\no{S\big[u_0, g_0, h_0; f\big]}_{C([0, T]; H_x^s(0, \ell))}
&\leqslant
c_{s, p}
\Big(
\no{u_0}_{H_x^s(0, \ell)}
+
\no{g_0}_{H^{\frac{2s+1}{4}}_t(0, T)}
+
\no{h_0}_{H^{\frac{2s+1}{4}}_t(0, T)}
\label{nh-fi-flh-ibvp-se-l}
\\
&\hskip 1.3cm
+
\sqrt T \prod_{j=1}^p  \no{f_j}_{C^\alpha([0, T]; L_x^{p}(0, \ell))}\Big),
\quad
0\leqslant s<\tfrac 12, 
\nn
}
}
the time estimates
\sss{\label{nh-fi-flh-ibvp-te}
\ddd{
\no{S\big[u_0, g_0, h_0; f\big]}_{C([0, \ell]; H_t^{\frac{2s+1}{4}}(0, T))}
&\leqslant
c_{s, p}
\Big(
\no{u_0}_{H_x^s(0, \ell)}
+
\no{g_0}_{H^{\frac{2s+1}{4}}_t(0, T)}
+
\no{h_0}_{H^{\frac{2s+1}{4}}_t(0, T)}
\label{nh-fi-flh-ibvp-te-h}
\\
&\hskip 1.3cm
+
\sqrt T \prod_{j=1}^p  \no{f_j}_{C([0, T]; H_x^s(0, \ell))}\Big),
\quad
\tfrac 12<s<\tfrac 32, 
\nn\\
\no{S\big[u_0, g_0, h_0; f\big]}_{C([0, \ell]; H_t^{\frac{2s+1}{4}}(0, T))}
&\leqslant
c_{s, p}
\Big(\!
\no{u_0}_{H_x^s(0, \ell)}
+
\no{g_0}_{H^{\frac{2s+1}{4}}_t(0, T)}
\!\!+
\no{h_0}_{H^{\frac{2s+1}{4}}_t(0, T)}
\label{nh-fi-flh-ibvp-te-l}
\\
&\hskip 1.3cm
+
\sqrt T \prod_{j=1}^p  \no{f_j}_{C^\alpha([0, T]; L_x^{p}(0, \ell))}\Big),
\quad
0\leqslant s<\tfrac 12,
\nn
}
}
and the $L^{p}$-estimate
\ddd{
\no{S\big[u_0, g_0, h_0; f\big]}_{C^\alpha([0, T]; L_x^{p}(0, \ell))}
&\leqslant
c_{s, p}
\Big(
\no{u_0}_{H_x^s(0, \ell)}
+
\no{g_0}_{H^{\frac{2s+1}{4}}_t(0, T)}
+
\no{h_0}_{H^{\frac{2s+1}{4}}_t(0, T)}
\label{nh-fi-flh-ibvp-ca-est}
\\
&\hskip 1.3cm
+
T^{\alpha}
\prod_{j=1}^p \no{f_j}_{C^{\alpha}([0, T]; L_x^{p}(0, \ell))}\Big),
\quad
\tfrac 12-\tfrac{1}{p} < s<\tfrac 12,
\nn
}
where $c_{s, p}>0$ is a constant that depends only on $s$ and $p$.
\end{theorem}

\noindent
\textbf{Local well-posedness on the finite interval.}
The linear estimates of Theorem \ref{nh-fi-flh-ibvp-wp-t} allow us to carry out  \textit{exactly} the same contraction mapping argument with that of Section \ref{nh-lwp-s} for the half-line in order to establish Theorems \ref{nh-fi-ibvp-wp-h-t} and \ref{nh-fi-ibvp-wp-l-t} for the local well-posedness of the nonlinear IBVP \eqref{nh-fi-ibvp} on the finite interval. Indeed, the only modification required in the proofs of Section \ref{nh-lwp-s} is the replacement of the solution spaces $X$ and $Y$ and of the data norm $\no{\cdot}_D$ by their finite interval counterparts as stated in Theorems \ref{nh-fi-ibvp-wp-h-t} and \ref{nh-fi-ibvp-wp-l-t}.

\vskip 3mm
%
%
%
%%%%%%%%%%%%%%%%%%%%%%
%
%		Acknowledgements
%
%%%%%%%%%%%%%%%%%%%%%%  
%
\noindent
\textbf{Acknowledgements.} The first author was partially supported by a grant from the Simons Foundation (\#524469 to Alex Himonas).

%
%
%%%%%%%%%%%%%%%%%%%%%%%%  
%
%			Bibliography
%
%%%%%%%%%%%%%%%%%%%%%%%%  
%
%
%

\vspace{8mm}
\noindent
A. Alexandrou Himonas \hspace{1.7cm} Dionyssios Mantzavinos  \hfill Fangchi Yan\\
Department of Mathematics \hspace{1.02cm} Department of Mathematics \hfill Department of Mathematics\\
University of Notre Dame \hspace{1.43cm} University of Kansas \hfill University of Notre Dame\\
Notre Dame, IN 46556 \hspace{1.99cm} Lawrence, KS 66045 \hfill Notre Dame, IN 46556 \\
E-mail: \textit{himonas.1$@$nd.edu} \hspace{1.28cm} E-mail: \textit{mantzavinos@ku.edu} \hfill E-mail: \textit{fyan1@nd.edu} 

\end{document}